\documentclass[preprint,10pt]{imsart}

\RequirePackage[OT1]{fontenc}
\RequirePackage{amsthm,amsmath}
\RequirePackage[numbers]{natbib}
\RequirePackage[colorlinks,citecolor=blue,urlcolor=blue]{hyperref}

\usepackage[dvips]{graphicx}
\usepackage{verbatim}
\usepackage{bbm}
\usepackage{graphicx} % standard LaTeX graphics tool
\usepackage{subfig}
\usepackage[cmex10]{amsmath}
\usepackage{amssymb}
\usepackage{mathrsfs}
\usepackage{verbatim}
\usepackage{epstopdf}

\startlocaldefs
\numberwithin{equation}{section}
\theoremstyle{plain}
\endlocaldefs

\newcommand{\bN}{\mathbb{N}}

\newcommand{\essup}{ \text{essup}}

\newcommand{\cJ}{\mathcal{J}_\cB}
\newcommand{\cJN}{\mathcal{J}_\cN}

\newcommand{\cA}{\mathcal{A}}
\newcommand{\cE}{\mathcal{E}}
\newcommand{\cB}{\mathcal{B}}
\newcommand{\cC}{\mathcal{C}}

\newcommand{\cF}{\mathcal{F}}

\newcommand{\cN}{\mathcal{N}}
\newcommand{\cI}{\mathcal{I}}

\newcommand{\cH}{\mathcal{H}}
\newcommand{\cP}{\mathcal{P}}

\newcommand{\ccg}{\mathcal{C}_{\gamma}}

\newcommand{\cFt}{\mathscr{F}_{t}}

\newcommand{\calo}{{\mathcal{O}}}
\newcommand{\Exp}{{\sf E}}

\newcommand{\Pro}{{\sf P}}

\newcommand{\ind}[1]{\mathbbm{1}_{\{#1\}}}

\newtheorem{definition}{Definition}

\newtheorem{corollary}{Corollary}
\newtheorem{lemma}{Lemma}
\newtheorem{remark}{Remark}
\newtheorem{theorem}{Theorem}

\begin{document}

\begin{frontmatter}
\title{Efficient Byzantine \\Sequential Change Detection}
\runtitle{Efficient Byzantine Sequential Change Detection}
%\thankstext{T1}{E.B. is  supported in part by the US National Science Foundation under DMS-1613170. L.L. is  supported in part by the US National Science Foundationunder CNS-1457076.}
\dedicated{Dedicated to H. Vincent Poor on the occasion of his 65th Birthday.}

\begin{aug}
\author{\fnms{Georgios} \snm{Fellouris}\thanksref{a,e1}\ead[label=e1,mark]{fellouri@illinois.edu}}
\author{\fnms{Erhan} \snm{Bayraktar}\thanksref{b,e2}\ead[label=e2,mark]{erhan@umich.edu}}
\and
\author{\fnms{Lifeng} \snm{Lai}\thanksref{c,e3}\ead[label=e3,mark]{lflai@ucdavis.edu}}

\thankstext{a}{G.F. is  supported in part by the US National Science Foundation
under CIF-1514245}
\thankstext{b}{E.B. is  supported in part by the US National Science Foundation under DMS-1613170}
\thankstext{c}{L.L. is  supported in part by the US National Science Foundation
under CNS-1457076}

\address[a]{Department of  Statistics,  Coordinated Science Lab \\ University  of Illinois, Urbana-Champaign, \\
IL 61820, USA. \printead{e1}}

\address[b]{Department of  Mathematics \\
 University of Michigan, Ann Arbor \\
 MI  48109, USA   \printead{e2}}

\address[c]{Department of Electrical and Computer Engineering \\
 University of California, Davis \\
CA 95616, USA \printead{e3}}

\runauthor{Fellouris, Bayraktar and Lai}
\affiliation{University of}

\end{aug}

\begin{abstract}
In the multisensor sequential change detection problem, a disruption occurs in an environment monitored by multiple sensors. This disruption induces a change in the observations of an unknown subset of sensors. In the Byzantine version of this problem, which is the focus of this work, it is further assumed that the postulated change-point model may be misspecified for an unknown subset of sensors. The problem then is to detect the change quickly and reliably, for any possible subset of affected sensors, even if the misspecified sensors are controlled by an adversary.  Given a user-specified upper bound on the number of compromised sensors, we propose and study three families of sequential change-detection rules for this problem. These are designed and evaluated under a generalization of Lorden's criterion, where conditional expected detection delay and expected time to false alarm are both computed in the worst-case scenario for the compromised sensors. The first-order asymptotic performance of these procedures is characterized as the worst-case false alarm rate goes to 0. The insights from these theoretical results are corroborated by a simulation study. 
\end{abstract}



\end{frontmatter}

\section{Introduction}\label{sec:Intro}

Suppose  that a system   is monitored in real time by multiple sensors that communicate with a fusion center. At an unknown time, a disruption occurs and  induces a change in the observations of a  subset of deployed sensors.  In this context, the  multisensor (or multichannel) sequential (or quickest) change detection problem is to combine at the fusion center  the information from all  sensors  in order to detect the change as soon as possible, while controlling the rate of false alarms. This problem has been studied extensively when the change is perceived by exactly \textit{one} unknown sensor \cite{tarta_steklov,khan,tarta_cdc,blazek,blazek2,tar_veer}. 
The assumption of a unique affected sensor has been removed in various recent works, where the change is allowed  to affect an arbitrary, unknown subset of sensors \cite{mei_bio,xie,greg_paper}. In this context,  even in the absence of any information regarding the affected sensors, 
it is  possible to  achieve  the optimal detection performance, in Lorden's  sense \cite{lorden1}, 
up to a first-order  asymptotic approximation \cite{mei_bio}, or even up to a constant term \cite{greg_paper}, 
as the false alarm rate goes to 0.   In a  related line of research, the  affected sensors perceive the change at possibly different times \cite{MR2299631, rag,oly,lud,oly2}, and the goal is to detect the first of these times. 
 
In all these references  it is assumed that the observations in the unaffected sensors  continue to behave in the same way as before the disruption takes place. 
%The performance loss inflicted by the presence of such unaffected sensors is  ``controllable''. 
 A different formulation, inspired by the Byzantine generals problem in fault tolerant design \cite{byzantine-generals-problem},  was considered in \cite{bay_lai}. There, it was assumed
that  the change is perceived by all  sensors \textit{apart from a single, unknown sensor that is compromised}, in the sense that its   observations are generated by an adversary. %whose goal is to prevent the detection of the change. 
This  formulation is  motivated by the interest in designing fault tolerant quickest detection schemes in   security related scenarios, where an adversary might take control of certain deployed sensors in order to foil change-detection schemes.  We can also think of this  formulation  as a  robust version of the classical multichannel sequential change-detection problem, where the postulated model in some sensors is not  correctly specified. 
% not only be unaffected by the change, but may malfunction in an unpredictable way.
%However, it can also be considered as a robust version of the classical multichannel problem. 

The solution that was proposed in  \cite{bay_lai}  for this
problem was a decentralized \textit{second alarm}, where   each sensor computes  its local Cumulative Sums (CUSUM) statistic \cite{page}, raises an alarm as soon as this statistic exceeds a fixed threshold, and the  fusion center  stops as soon as \textit{two} distinct sensors have raised an alarm.   This scheme  was  analyzed  under a  generalized version of Lorden's criterion, where the worst case scenario is considered for the change-point,  the observations in up to the change-point,  \textit{and also the strategy of the adversary}. In this context  it was shown that the  worst-case detection delay of the second alarm grows logarithmically with its worst-case expected time to false alarm, a property that is not preserved by the centralized CUSUM  rule that assumes that all sensors are honest and affected. Moreover, it was shown that the  performance  of the second alarm can be significantly improved if it is applied, in a centralized fashion, to three disjoint \textit{groups} of sensors. However, the asymptotic performance  of these two schemes was not  characterized,  and  neither of them was shown to be efficient or optimal in any sense.

In this work we generalize the \textit{Byzantine} sequential change detection proposed  in \cite{bay_lai} in two ways.  First,  we allow for more than one sensors to be compromised, and  second, we allow for only an unknown subset of honest sensors to be affected by the change. Thus, we have three categories of sensors: the corrupt (or compromised) ones, which are assumed to be controlled by an adversary,  the honest ones that perceive the change, and the honest ones that do not perceive the change. 

In this context, we propose and study three families of multichannel sequential change-detection schemes that require the computation of the  local CUSUM statistics at the sensors and are parametrized by a number $L$ that takes values between $1$ and $K$, the total number of sensors.  In the first  one, the fusion center raises an alarm when $L$ individual local CUSUM statistics have crossed a common threshold, and we refer to it as the \textit{$L^{th}$ alarm};  the second alarm proposed in \cite{bay_lai} is a special case of this family $(L=2)$.  In the second,  the fusion center stops when $L$  sensors  agree that the change has occurred, in the sense that their corresponding local CUSUM statistics are  simultaneously above a common threshold;  we refer to it as  \textit{voting rule}, since it requires from each sensor to ``vote'' at each time whether the change has occurred. In the  third one, the fusion center stops when the sum of the $L$ smallest local CUSUM  statistics exceeds a threshold,  and we  refer to  it  as \textit{Low-Sum-CUSUM}.

We assume that there is  a known,  user-specified upper bound,  $M$, on the number of compromised sensors, and  we design and analyze the proposed schemes under a   generalization of Lorden's criterion, where the conditional expected detection delay and expected time to false alarm are  evaluated  \textit{in  the  worst case regarding the strategy of the adversary when there are \textit{exactly}   $M$  corrupt sensors}.
% both before and after the change, in $M$ corrupt sensors. % The design of the proposed schemes and their comparisons will be based on this  worst-case setup, in which it will only be possible to detect the change when the number of honest sensors affected by the change is larger than $M$. 

%However,  in  \cite{bay_lai}, we assume that the observations in the honest sensors are independent with the same density before the change, and and identically distributed before/after the change, and the honest sensors are independent and homogeneous. However,

The main contributions of this work are the following: first of all,
for each family of detection rules under consideration, we obtain  bounds on $L$  so that the resulting rules can  control the worst-case false alarm rate below an arbitrary, user-specified level, and  achieve non-trivial detection performance whenever the size of 
the affected subset is larger than $M$. 
%These bounds on $L$ coincide for the three families in the special case that all honest sensors are affected by the change.
%, in which case $L$ must be larger than $M$ and at most equal to $K-M$, i.e., $M+1 \leq L \leq K-M$. 
A particular case of interest arises when $K-M=M+1$. In this case, the only possible value for $L$   is $M+1$ in all schemes, and  the  three resulting procedures  are shown to be  strictly ordered. Specifically,   \textit{for any given false alarm rate}, the $(M+1)$-alarm is shown to be strictly better than the corresponding voting rule, and the latter  strictly better than  \textit{Low-Sum-CUSUM}.

In order to select the parameter $L$ and compare the proposed procedures when $K-M>M+1$, we conduct  an asymptotic analysis   and   characterize their performance  up to a first-order asymptotic approximation as the worst-case false alarm rate goes to 0.
%the resulting rule are able to detect the change when the size of 
%the affected subset, $|\cB|$, is larger than $M$ %As a by-product, we obtain a  first-order asymptotic approximation to the performance of the second alarm  considered  in \cite{bay_lai}.  
These  asymptotic results suggest that the most appropriate choice for $L$ is $M+1$  for the $L^{th}$ alarm and $K-M$ for   \textit{Low-Sum-CUSUM}. With this selection of $L$,  the first-order asymptotic detection delay of  \textit{Low-Sum-CUSUM} is $|\cB|-M$ times smaller than that of the $(M+1)$-alarm, where $\cB$ is the affected subset and $|\cB|$ its size. 

 On the other hand, the proposed choice for $L$, and the resulting  asymptotic performance, in  the family of \textit{voting rules} depends on the prior knowledge regarding the size of the affected subset, $|\cB|$. In the absence of any information, we suggest setting $L=M+1$, in which case the resulting voting rule has the same first-order asymptotic performance as the $(M+1)$-alarm. When $|\cB|$ is known in advance, we suggest setting  $L=|\cB|$, in which case the resulting voting rule has the same first-order asymptotic performance as   \textit{Low-Sum-CUSUM}.

Finally, similarly to   \cite{bay_lai},  
when $K>2M+1$ we also consider a centralized version of the  $(M+1)$-alarm and quantify its asymptotic performance  when
it is known in advance that  all  honest sensors are affected. In this setup, we show that the first-order asymptotic performance of the  centralized $(M+1)$-alarm is  $K/(2M+1)$ smaller than that of the  decentralized $(M+1)$-alarm.  The resulting performance however is inferior to that of the voting rule and \textit{Low-Sum-CUSUM}  with $L=K-M$. 

   These theoretical results are  supported by two  simulation studies, where all honest sensors are affected and there is only one compromised sensor $(M=1)$.  In the first one, the number of honest sensors   exceeds by 1 the number of corrupt sensors, and as predicted from our theoretical results,  the $(M+1)$-alarm is shown to perform better than the voting rule, and the latter to perform better than  \textit{Low-Sum-CUSUM}. In the second simulation study, the number of honest sensors is larger than $M+1$, in which case the previous order is completely reversed. Moreover, the centralized $(M+1)$-alarm performs  significantly better than the decentralized $(M+1)$-alarm, but, as expected, worse than  the voting rule and \textit{Low-Sum-CUSUM} with $L=K-M$. 

The rest of the paper is organized as follows: in Section \ref{sec:formulation} we formulate the problem mathematically. In Section \ref{sec:oracle}  we study various CUSUM-based,  sequential change-detection procedures in  the classical multichannel formulation where there are no corrupt sensors, or equivalently  the honest sensors are known in advance.  In Section \ref{sec:robust}  we introduce and study the proposed procedures in the presence of corrupt sensors. 
In Section \ref{sec:simulation},  we present the results of two simulation studies.  We conclude in Section \ref{sec:conclude}.  

 In terms of  notation, we denote by  $\Theta(1)$ a non-zero constant term, by $\calo(1)$  a bounded from above term,  and by $o(1)$ a vanishing term in a limiting sense that will be understood by the context. We set $x\sim y$ when $x/y \rightarrow 1$, $x \leq y (1+o(1))$ when  $\limsup (x/y) \leq 1$,
 $x \geq y (1+o(1))$ when  $\liminf (x/y) \geq 1$. Moreover, we
 set   $x^{+}=\max\{x,0\}$ and we use $|\cdot |$ to denote the size of a set.

\section{Problem formulation}\label{sec:formulation}
Suppose data are collected  sequentially from $K$ sensors.
For each $k \in [K]$,  let $X^k \equiv \{X_t^k\}_{t \in \bN}$ be the sequence of observations in the $k^{th}$ sensor, where $\bN=\{1, 2, \ldots\}$ and $[K]= \{1, \ldots, K\}$. We assume that there is a  subset $\cN \subseteq [K]$ of independent sensors, in the sense that 
$\{X^k, k \in \cN\}$ are independent sequences. For each $k \in \cN$,  $X^k$ is a sequence of independent random variables,  which are initially distributed according to some density $f$. This density changes 
in a subset of  sensors, $\cB \subseteq \cN$, 
 at some  unknown, deterministic point in  time $\nu \in \{0,1, \ldots\}$ (\textit{change-point}). That is,
\begin{align} \label{model}
\begin{split}
X_t^{k} &\sim f, \quad t \in \bN, \quad k \in \cN \setminus \cB, \\
X_t^{k} &\sim
\begin{cases}
 f, \quad t \leq \nu   \\
 g, \quad   t > \nu,
 \end{cases} k \in \cB.
 \end{split}
\end{align}  
We assume that  $f$ and $g$ are known densities with respect to a $\sigma$-finite measure $\lambda$ and denote by $\cI$ their  Kullback-Leibler information number:
\begin{align} \label{first}   
0 < \cI  \equiv  \int \log  \left(\frac{g(x)}{f(x)}\right) \,  g(x) \, \lambda(dx) < \infty. 
\end{align}
This is a  homogeneous change-point model, in the sense that the pre-change densities are the same in all sensors, and so are the post-change densities in those sensors that perceive the change.  We assume that $\mathcal{I}>0$ without loss of generality,  since otherwise $f=g$ $\lambda$-almost everywhere, and also  $\mathcal{I} <\infty$, which is a standard assumption in the asymptotic analysis of   sequential change-detection procedures (see e.g.,  \cite{lorden1}).  We will refer to certain results in the literature that require the following second-moment assumption:
\begin{align} \label{second}   
\int \left(\log  \left(\frac{g(x)}{f(x)}\right) \right)^2\,  g(x) \, \lambda(dx)<\infty. 
\end{align}
However,  our standing assumption throughout the paper  is that   \eqref{first} holds, and  will not be stated explicitly from now on.
 
  Since the change-point is unknown and observations are collected sequentially, the problem is to find  a stopping rule that determines when to 
stop and declare that the change has occurred based on the data from all sensors. More formally, 
a  \textit{sequential change-detection rule} is an  $\{\cFt\}$-stopping time, where  $\cFt$ is the $\sigma$-field generated by the observations in all sensors up to time $t$, i.e., 
$$
\cFt=\sigma\left(X_s^k: 1\leq s \leq t, k \in [K] \right).
$$ 
Our goal  is to  propose detection rules that are able to detect the change quickly and reliably  \textit{even in the worst-case scenario that  the sensors that do \textit{not} belong to $\cN$ are  controlled by an adversary who tries  to prevent the detection of the change}. Thus, we  refer to the sensors in $\cN$ as \textit{honest}, and to those  that do not belong to $\cN$ as \textit{corrupt, or compromised}. Moreover, we refer to  the sensors in $\cB \subseteq \cN$ as \textit{affected}  (by the change), and to the ones in $\cN \setminus \cB$ as \textit{unaffected} (by the change).

We assume that there are at most $M$ corrupt sensors, where $M \geq 0$ is a user-specified number that can be determined based on prior information on the quality of each sensor. Alternatively, we can think of $M$ as a parameter that  represents the amount of robustness that  we want to introduce to the classical multichannel problem, which corresponds to the case
 $M=0$. When $M=1$ we recover  the setup considered in \cite{bay_lai}.   We consider  the worst possible scenario regarding  the \textit{number} of corrupt sensors and    the \textit{strategy of the adversary}. 

Indeed,  our analysis focuses on the case 
 that  there are exactly $M$ corrupt sensors,  and consequently   $|\cN|=K-M$ honest sensors, with the understanding that the proposed procedures will still be effective when the true number of corrupt sensors is smaller than $M$.   Since we consider a homogeneous change-point model,  we can assume without loss of generality that the subset of honest sensors, $\cN$, is an arbitrary subset of size $K-M$.
 This allows us to lighten the notation by  suppressing in what follows the dependence on $\cN$ of many quantities of interest.

In order to identify the worst-case regarding the strategy of the adversary, we assume that the latter  knows  the true change-point and affected subset,  and has the same access to the observations of the honest sensors as the fusion center. 
To be more specific,  let   $\cH_t$ 
%and $\cG_t$ 
denote  the $\sigma$-field generated by the observations  in the  honest sensors up to time $t$, i.e.
% and corrupt sensors respectively, i.e., 
\begin{align*}
\cH_t &=\sigma\left(X_s^k: 1\leq s \leq t, k \in \cN \right), %\\
%\cG_t &=\sigma \left(X_s^k: 1\leq s \leq t, k \notin %\cN \right),
\end{align*}
suppressing its dependence on $\cN$. Then, we assume that  the observations in the corrupt sensors at time $t$ form a $\cH_t$-measurable random vector, i.e.,  there  is a  Borel function 
$\pi_t$ so that 
$$
X_t^{\cN^c}=\pi_t\left(X_s^\cN, 1\leq s \leq t \right),
$$
where for each subset $\cC \subseteq [K]$ and time $t$ we use the following notation: $ X_t^\cC \equiv  (X_t^k, k \in \cC)$.  We define the strategy of the adversary as the family of deterministic functions
$$\pi \equiv \{\pi_t: t \in \bN\},$$ 
suppressing its dependence on $\cN$,  $\nu$ and  $\cB$.  We will consider the worst-case scenario regarding the  strategy of the adversary when we 
 evaluate both the detection delay and the false alarm rate of a procedure. 

We  denote the probability measure in the underlying canonical space  by  $\Pro_\nu^{\cB,\pi}$ when the change occurs at time $\nu$ in a subset $\cB \subseteq \cN$ of honest sensors and the strategy of the adversary is $\pi$, with the understanding that under  $\Pro_\infty^{\pi}$ there is no change in the honest sensors.
We simply write  $\Pro^{\cB}_\nu$ and $\Pro_\infty$ instead of  $\Pro_\nu^{\cB,\pi}$  and  $\Pro_\infty^{\pi}$ when the event of interest depends only on observations  from the honest sensors.  Following Lorden's \cite{lorden1} approach, we  quantify the delay of a sequential change-detection rule $T$  when the change occurs in subset $\cB$ with the following criterion:  
 $$
\cJ[T] = \sup_{\nu, \pi}  \;  \essup \; \Exp_\nu^{\cB,\pi} \left[ \left(T-\nu\right)^{+}| \cF_\nu \right],
 $$ 
 where $\Exp_\nu^{\cB, \pi}$ is expectation under $\Pro_\nu^{\cB,\pi}$.  Thus,  we consider the worst-case scenario with respect to  change-point $\nu$ and the observations until the time of the change,  as in Lorden's criterion, but now we also consider the worst-case scenario regarding  the strategy of the adversary in the subset of  $M$ corrupt sensors that it controls. We also take a worst-case approach in the quantification of the expected time to false alarm, which we define as follows: 
 $$\cA[T] = \inf_{\pi} \Exp_{\infty}^{\pi}[T].$$
 We denote by $\ccg$ the class of sequential change-detection rules for which the worst-case expected time to false alarm is bounded below by some  user-specified constant $\gamma >1$, i.e., 
$\ccg =\{T :  \cA[T]  \geq \gamma\}$.  We are interested in designing sequential change-detection rules  that belong to    $\ccg$ for some arbitrary  $\gamma >1$, and at the same time have ``small''  worst-case detection delay, $\cJ$, \textit{for ideally every possible affected subset, $\cB$}.  This will turn out to be possible for the proposed procedures  only when the size of the affected subset, $|\cB|$, is  larger than $M$.  With this in mind, we introduce the following notion of domination in order to compare detection rules in our context.  \\

%In this work, we assume an upper bound on the number of compromised sensors, $M \geq 1$. Since both $\cJ$ and $\cA$ consider the worst case regarding the corrupt sensors, the design and analysis of the proposed schemes will be based on the  case that there are \textit{exactly} $M$ affected sensors.  

\begin{definition} \label{def1}
Let $T$ and $S$ be two  multichannel sequentially change-detection rules. We say that $S$ \textit{dominates}  $T$ if  $\cJ[T] \geq\cJ[S]$ for every $\cB \subseteq \cN$  so that $M+1 \leq |\cB| \leq K-M$ whenever   $\cA[T] \leq \cA[S]$. \\
\end{definition}

Such a strict domination property will arise  only in the  special case that $K=2M+1$.  In  general, our comparisons will rely on asymptotic approximations as the worst-case false alarm rate goes to 0,  which leads to  the following definition. \\

\begin{definition} \label{def2}
Let $T$ and $S$ be  
multichannel sequential  change-detection rules. We say that $S$ \textit{is asymptotically more efficient than  $T$} if  
$\cJ[T] \geq \cJ[S] \, (1+o(1))$ for every $\cB \subseteq \cN$ so that $M+1 \leq |\cB| \leq K-M$  as $\cA[T] = \cA[S] \rightarrow \infty$. 
\end{definition}

% for every possible affected subset $\cB$. Moreover, we will say that $T$ is \textit{robust}  if  $\cJ[T]$ grows logarithmically with $\cA[T]$  for every possible affected subset $\cB$. 
%Our main goal in this work is to propose robust and efficient sequential change-detection rules in the presence of corrupt sensors, which is done in Section \ref{sec:robust}. 

%\begin{equation}\label{always}
%|\cN^c|=M < K-M =|\cN|,
%\end{equation}
 %i.e., the number of corrupt sensors is smaller than the  number of honest sensors, and consequently $K>2M$. 

\section{The classical multichannel setup}\label {sec:oracle} 
In this section  we consider  the classical multichannel framework where  there are no corrupt sensors, or equivalently the subset, $\cN$, of honest sensors is known in advance, but the subset of affected sensors, $\cB \subseteq \cN$, is not.   The procedures and results of this section will provide the basis for the methods and analysis  in Section \ref{sec:robust} where $\cN$ will also be unknown. However, the results in this section  may also be of  independent interest for the classical multichannel problem itself, as  we revisit  various  multichannel, CUSUM-based schemes in the literature. 
% and we also propose a novel scheme (Subsection \ref{subsec:LowSumCUSUM}) that will turn out to be  particularly suitable in the presence of corrupt sensors. 

\subsection{Notation} \label{subsec:classical_notation}
For each $\cC \subseteq \cN$ and $t \in \bN$  we denote  by $Z_t^\cC$ the cumulative log-likelihood ratio of the first $t$ observations in the sensors in $\cC$, i.e.,
% $\ell_t^k$ the log-likelihood ratio of $X_t^k$ and by
\begin{align} \label{LLr}   
Z_t^\cC=Z_{t-1}^{\cC}+ \sum_{k \in \cC} \ell_t^k; \quad  Z_0^\cC \equiv 0,
\end{align}
where  $\ell_t^k$ is the log-likelihood ratio of the $t^{th}$ observation in sensor $k$, i.e., 
\begin{align} \label{llr}   
 \ell_t^k =\log \left( \frac{g(X_t^k)}{f(X_t^k)} \right). 
\end{align}
We denote by $W_t^\cC$ Page's \cite{page} CUSUM statistic at time $t$  for detecting a change in  subset $\cC \subseteq \cN$, i.e., 
\begin{align*} 
  W_t^{\cC}= \left( W_{t-1}^{\cC} + \sum_{k \in \cC} \ell_t^k \right)^{+} ; \quad W_0^\cC \equiv 0.
\end{align*}
 We denote by $\sigma_{\cC}(h)$  the corresponding CUSUM stopping time, that is the  first time the process $W^\cC$  exceeds a positive threshold $h$,   i.e., 
\begin{align} \label{opt} 
\sigma_{\cC}(h) = \inf\left\{t \in \bN: W_t^{\cC} \geq h\right\}.
 \end{align} 
 When $\cC=\{k\}$ for some $k \in \cC$, we simply write $Z^k_t$,   $W_t^k$ and $\sigma_{k}(h)$, instead of $Z_{t}^{\{k\}}$, $W_t^{\{k\}}$ and  $\sigma_{\{k\}}(h)$.  Moreover, we use  the following notation for the ordered local CUSUM stopping times and statistics: 
\begin{align}  \label{ordered_CUSUMs}
\begin{split}
& \sigma_{(1)}(h) \leq \ldots \leq \sigma_{(|\cN|)}(h) ,  \\
& W_t^{(1)} \leq \ldots \leq W_t^{(|\cN|)} .
\end{split}
  \end{align}

 \subsection{Centralized CUSUM and the optimal performance}
 It is useful for the subsequent development to recall some well-known properties of the centralized CUSUM stopping time,
$\sigma_{\cC}(h)$. For every $s \in \bN$  and $\cC \subseteq \cN$ we have  (see, e.g., \cite[Appendix 2]{sieg2}) that   
\begin{align} \label{domin} 
\Pro_\infty( W_s^\cC \geq h)  \leq e^{-h}.
  \end{align}  
In fact,   $\sigma_{\cC}(h) / e^h$ is asymptotically exponential (see, e.g., \cite{khan}),  and consequently  as $h \rightarrow \infty$
\begin{align} \label{opt_ARL} 
\Exp_{\infty} \left[\sigma_{\cC}(h) \right] \sim \Theta(1) \, e^{h}.
  \end{align}  
 Moreover, we have the following  decomposition of the CUSUM detection statistic
\begin{equation} \label{cusum_decomposition}
W_t^\cC= Z_t^\cC+  m_t^{\cC}, 
\quad m_t^{\cC}= -\min_{0 \leq s \leq t} Z_s^{\cC},
\end{equation} 
which implies that $W_t^\cC\geq  Z_t^\cC$ for every $t$.  In view of this decomposition,  from non-linear renewal theory \cite[Section 2.6]{TNB_book2014} 
it follows that  when  $\cC$ is included in the affected subset $(\cC\subseteq \cB)$, 
%, sincethe worst-case scenario for the observations  of the sensors in $\cC$ up to the time of the change $\nu$ is that $W^{\cC}_{\nu}=0$ and    $W^\cC$ regenerates whenever it hits 0. from  non-linear renewal theory \cite[Section 2.6]{TNB_book2014} it follows that
then as $h \rightarrow \infty$ 
\begin{equation} \label{cusum_as_approximation}
\Pro_{0}^{\cB} \left( \sigma_{\cC}(h) \sim \frac{h}{|\cC| \, \cI} \right) =1,
\end{equation} 
and consequently  for any $r \geq 1$ 
\begin{equation} \label{cusum_moment_approximation}
\Exp_{0}^{\cB} \left[\sigma^r_{\cC}(h) \right] \sim  \left( \frac{h}{|\cC| \, \cI} \right)^r.
\end{equation}
Since  $\cJ\left[\sigma_{\cC}(h)\right]= \Exp_{0}^{\cB} \left[\sigma_{\cC}(h) \right]$ for every $h>0$ (see, e.g. \cite{moust1}), if  $h=h_\gamma$ is so that $\Exp_{\infty} \left[\sigma_{\cC}(h_\gamma) \right]  =\gamma$, then 
  as $\gamma \rightarrow \infty$  
 \begin{align} \label{cusum_perf} 
\cJ \left[\sigma_{\cC}(h_\gamma) \right] \sim \frac{\log \gamma}{|\cC| \cI }.
  \end{align}  
When in particular $\cC=\cB$ and $h=h_\gamma$ is selected so that  $\Exp_\infty[\sigma_{\cB}(h_\gamma)]=\gamma$,  $\sigma_{\cB}$  optimizes $\cJ$ within the class of detection rules $\ccg$ \cite{moust1}. This optimality property, combined with \eqref{cusum_perf}, implies that  a first-order approximation to the  optimal  performance as $\gamma \rightarrow \infty$ is 
\begin{align} \label{opt_perf} 
\inf_{T \in \ccg}\cJ[T] \sim \frac{\log \gamma} {|\cB| \,\cI},
  \end{align}  
a result that was originally established in   \cite{lorden1} in a different way. We will refer to $\sigma_\cB$ as the \textit{optimal} or \textit{oracle} CUSUM, as it achieves the optimal performance but requires knowledge of the affected subset $\cB$.  

\subsection{Decentralized, multichannel detection rules}
When the subset of affected sensors is not known in advance,  it is desirable to design procedures that have ``good'' performance  under any possible affected subset. This is known to be possible even in the absence of any prior information regarding the affected subset.  Indeed, under the second moment condition \eqref{second},  the optimal performance is achievable  up to a constant term  under \textit{any} possible affected subset, e.g.,  by  the  \textit{GLR-CUSUM},  $ \min_{\cC \subseteq \cN} \sigma_{\cC}(h)$,  with $h=\log \gamma$ \cite{greg_paper}. %We will refer to this rule as \textit{GLR-CUSUM}, as it relies on maximizing the  CUSUM detection statistic with respect to the unknown affected subset. 
While this is a recursive rule, the number of recursions it requires grows exponentially with the number of honest sensors, $|\cN|$. %On the other hand, it admits  the following representation
%\begin{align} \label{GLR_cusum}
 %\inf\left\{t \in \bN: \max_{1 \leq s \leq t} \sum_{k \in \cN} (Z_t^k-Z_s^k)^+ \geq h\right\},
%\end{align}
%which is scalable with respect to the number of sensors, but its memory requirements can be very large when the false alarm rate is small.  
On the other hand, it is possible to achieve the optimal performance  for any possible affected subset,
 up to a \textit{first-order} approximation,  by  a procedure whose detection statistic is an increasing function of the local  CUSUM statistics (see Subsection \ref{subsec:sum_cusum}). 
 In this work, we focus on multichannel sequential procedures of this form, and the following lemma is useful for analyzing their worst-case detection delay. 

\begin{lemma} \label{lem:worst_case}
Let $\psi: [0,\infty)^{|\cN|} \rightarrow [0,\infty)$ be 
 a non-constant function that is  increasing in each of its arguments,
  and  consider the detection rule 
$$
S^*(h)=\inf\{t \in \bN: \psi\left( W_t^{1}, \ldots, W_t^{|\cN|} \right) \geq h\}.
$$ 
Then,   $\cJ\left[S^*(h)\right]= \Exp_{0}^{\cB} \left[S^*(h) \right]$ for every $h>0$.
\end{lemma}

\begin{proof} 
The worst-case scenario for the observations up to the time of the change $\nu$ is that $W^{k}_{\nu}=0$ for every $k \in \cN$. 
Under $\Pro_\infty$, the process $(W^1, \ldots, W^{\cN})$ is a Markov chain that regenerates whenever all its components are equal to  0, which completes the proof. 
\end{proof}

\subsection{One-shot  schemes and voting rules}
Let $1 \leq L \leq |\cN|$.  We refer to $\sigma_{(L)}(h)$, 
defined in \eqref{ordered_CUSUMs},  as the \textit{$L^{th}$ honest alarm}, since it represents the first time the  local CUSUM statistics in $L$ honest sensors have crossed level $h$. A related stopping rule  is
\begin{align} %\label{general}
  S_{L}(h) &=\inf\left\{t \in \bN: W_t^{(|\cN|-L+1)}  \geq h \right\},  
 \end{align}
 which is the first time the local CUSUM statistics in  $L$ honest sensors are \textit{simultaneously} above $h$;  we  will refer to it  as  \textit{voting rule}, since  it requires from each sensor to ``vote''  at each time whether the change has occurred or not. 
 In general, we have  $\sigma_{(L)}(h) \leq S_{L}(h)$ for every $h>0$ and $1 \leq L \leq |\cN|$, with equality when $L=1$.  The \textit{first-alarm}, $\sigma_{(1)}$,   achieves  the optimal asymptotic performance \eqref{opt_perf}  when $|\cB|=1$, i.e., when exactly one sensor is affected by the change \cite{blazek,blazek2}.  The ``consensus'' rule, $S_{|\cN|}$, which stops when all honest local CUSUM statistics are simultaneously above a common threshold, is also known to be asymptotically optimal when \textit{all}  honest sensors are affected by the change ($\cB=\cN$) \cite{mei1}. These two results were shown recently \cite{sourabh_conf} to be special cases of a more general result, according to which the voting rule, $S_L$, achieves the optimal first-order asymptotic performance  \eqref{opt_perf}  when $L$ is equal to the size of the affected subset, $|\cB|$. In  Theorem  \ref{th:voting_optimality} we establish this result under only the first moment condition \eqref{first}, removing the  second moment condition  \eqref{second} that was assumed in both  \cite{mei1} and  \cite{sourabh_conf}. \\

% When  $L=|\cN|$, we will refer to $S_{L}$ as  the  \textit{consensus rule} \cite{mei1} , as it requires agreement from all sensors that the change has occurred.  This rule had been shown in \cite{mei} to achieve the optimal asymptotic performance \eqref{opt_perf} in the case that all sensors  These two classes of detection rules have been studied in a unified way in \cite{sourabh_conf}.
 %In order to state them,  we introduce the following notation:$$D_{L:|\cN|}= \sum_{i=|\cN|-L+1}^{|\cN|} (1/i),  \quad 1 \leq L \leq |\cN|. $$

%Of course,   the $L^{th}$ honest alarm is a meaningful procedure only when $L$ does not exceed the  size of the affected subset, $\cB$.    These rules have been studied in \cite{sourabh_conf},  where it was shown that the first honest alarm is the more efficient member in this class, although they all achieve the same, first-order asymptotic performance. The following lemma, whose proof can be found in \cite{sourabh_conf}, provides asymptotic approximations to the operating characteristics of the $L^{th}$ honest alarm. For this purpose we introduce the following notation:

\begin{lemma}  \label{lem: ARL_one_shot_and_voting}
Let $1 \leq L \leq |\cN|$.  Then as $h \rightarrow \infty$ 
\begin{equation} \label{ARL_voting} 
 \Exp_{\infty} \left[S_{L}(h) \right] \geq \Theta(1) \, e^{L h},
\end{equation} 
\begin{equation} \label{ARL_one_shot} 
 \Exp_{\infty} \left[\sigma_{(L)}(h) \right] \sim  \Theta(1)  \, e^{h}.
\end{equation}
\end{lemma}

\begin{proof}
The  lower bound in \eqref{ARL_voting} was established in \cite[Theorem 3.2]{sourabh_conf}.  The asymptotic approximation in \eqref{ARL_one_shot} follows from the asymptotic exponentiality \cite{khan} of the independent stopping times $\sigma_k(h)$, $1 \leq k \leq  |\cN|$. \\
\end{proof}

It is clear that for both $\sigma_{(L)}$ and $S_L$ to have  non-trivial detection performance, $L$ needs to be at most equal to the size of the affected subset, $|\cB|$. 
Indeed, when $L>|\cB|$, at least one of the $L$ alarms needs to come from an unaffected sensor, 
and consequently  $S_{L}(h) \geq \sigma_{(L)}(h) \geq \min_{k \in \cN \setminus \cB} \sigma_{k}(h)$, which means that the expected detection delay of $\sigma_{(L)}$ and   $S_L$ will be larger than the expected time to false alarm from at least one of the unaffected sensors.  Thus, the following lemma describes the asymptotic detection delay of  $\sigma_{(L)}$ and $S_L$ when $L\leq |\cB|$. \\

\begin{lemma}
If $1 \leq L \leq |\cB|$, then as $h \rightarrow \infty$
\begin{align} \label{ADD_one_shot} 
 \Exp_0^{\cB}\left[\sigma_{(L)} (h)\right]  
\sim h/\cI 
\sim \Exp_0^{\cB}\left[S_{L}(h)\right].
\end{align}
%When in particular $\cB=\cN$, the latter asymptotic upper bounds are sharp  in the sense that   as $h \rightarrow \infty$
%\begin{equation} \label{ADD_one_shot_all_affected} 
%\cJN\left[\sigma_{(1)} (h)\right] \sim \cJN\left[\sigma_{(L)} (h)\right] \sim 
%\frac{h}{\cI} \sim   \cJN\left[S_{L}(h)\right] \sim \cJN\left[S_{|\cN|}(h)\right].
%\end{equation} 
\end{lemma}

\begin{proof}
For every  $h>0$ and $ 1 \leq L \leq |\cB|$ we clearly have 
$\sigma_{(1)} (h)  \leq \sigma_{(L)}(h) \leq S_{L}(h) \leq S_{|\cB|}(h)$.   Therefore, it suffices to show that as $h \rightarrow \infty$
\begin{align}\label{show}
\begin{split}
 \Exp_0^{\cB}[ S_{|\cB|}(h) ] &\leq  \frac{h}{\cI} \; (1+o(1)),  \\ \Exp_0^\cB \left[\sigma_{(1)} (h)\right]  &\geq \frac{h}{\cI} \; (1+o(1)). 
 \end{split}
\end{align}
For every $h >0$ we have 
\begin{align*}
 S_{|\cB|}(h) &\leq \inf\{t \in \bN: W_t^k \geq h, \quad  \forall  \;  k \in \cB\}  \\
 &\leq \inf\{t \in \bN: Z_t^k \geq h,  \quad \forall  \;    k \in \cB\}.
\end{align*}
%i.e.,  $S_{|\cB|}(h)$ stops before the first time a number of random walks with drift $\cI$ are simultaneously above a threshold $h$.  
The asymptotic upper bound in \eqref{show} then follows from Lemma \ref{lem:simultaneous} in the Appendix.  

It remains to prove the asymptotic lower bound in \eqref{show}. To this end, we observe that $\sigma_{(1)}$ can be represented as the minimum of  a stopping time that perceives the change and one that does not. Specifically: 
$$
\sigma_{(1)}(h)= \min \left\{ \min_{k \in \cB} \sigma_{k}(h), \; \min_{k \notin \cB} \sigma_{k}(h) \right\}.
$$
It then suffices to show that  as $h \rightarrow \infty$ we have
\begin{align}
\Exp_0^\cB \left[\sigma_{(1)} (h)\right] &=
\Exp_0^{\cB}\left[ \min_{k \in \cB} \sigma_{k}(h) \right] -o(1) \label{ren0}\\
\Exp_0^{\cB}\left[ \min_{k \in \cB} \sigma_{k}(h) \right] &\geq \frac{h}{\cI} \, (1+o(1)).  \label{ren00}
\end{align} 
In order to prove \eqref{ren0} we  rely on Lemma \ref{newlemma}(ii) in the Appendix.  From Boole's inequality, for every $t=0,1, \ldots$, we have 
\begin{align*}
\Pro_\infty \left( \min_{k \notin \cB} \sigma_{k}(h)  \leq t \right) 
&\leq \sum_{k \notin \cB} \Pro_\infty \left( \sigma_{k}(h)  \leq t \right)  \\
&= \sum_{k \notin \cB} \Pro_\infty \left( \max_{1 \leq s \leq t} W_s^k \geq h \right)  \\
&\leq \sum_{k \notin \cB} \sum_{s=1}^{t} \Pro_\infty(W_s^k \geq h) \\
&\leq t \;  (|\cN|-|\cB|) \,e^{-h},
\end{align*}
where the last inequality follows from \eqref{domin}. 
Moreover, 
setting $r=2$ in  \eqref{cusum_moment_approximation} we have  
$$
\Exp_0^{\cB}\left[ \min_{k \in \cB} \sigma^2_{k}(h) \right] \leq   \frac{h^2}{\cI^2} \; (1+o(1)).
$$
%where the asymptotic approximation follows by setting $r=2$ in  \eqref{cusum_moment_approximation}. 
In view of Lemma  \ref{newlemma}(ii) in the Appendix, these two inequalities prove \eqref{ren0}. Finally, we obtain  \eqref{ren00}    from
\begin{align*} 
\Pro_0^{\cB} \left( \min_{k \in \cB} \sigma_{k}(h)/h   \underset{h \rightarrow \infty} \longrightarrow 1/ \cI \right)=1
\end{align*}
and   Fatou's lemma,   the former following from  \eqref{cusum_as_approximation}. \\
\end{proof}

Based on these two lemmas, we can now show that the first-order asymptotic performance of the $L^{th}$ honest alarm is independent of $L$. This result was shown in \cite{sourabh_conf}  under the second-moment condition \eqref{second}, which is  removed  in the following theorem.\\

\begin{theorem}\label{th:one_shot_optimality}
Let  $1 \leq L  \leq  |\cN|$. If    $h=h_\gamma$ so that  $ \Exp_\infty[\sigma_{(L)}(h_\gamma)]=\gamma$, then as $\gamma \rightarrow \infty$
 \begin{equation} \label{ARL_one_shot_approximation}
h_\gamma \sim   \log \gamma.
\end{equation}
If also   $L \leq |\cB|$, then  
 \begin{equation} \label{ADD_one_shot_approximation}
  \cJ\left[\sigma_{(L)} (h_\gamma)\right] \sim  (\log \gamma)/\cI .
\end{equation}
\end{theorem}

\begin{proof}
We obtain \eqref{ARL_one_shot_approximation} directly from \eqref{ARL_one_shot}. We obtain  \eqref{ADD_one_shot_approximation} by setting $h=h_\gamma$ in \eqref{ADD_one_shot} and  the fact that worst-case scenario for the change-point is $\nu=0$. \\
\end{proof}

\begin{remark} \label{rem:one_shot_choice_L}
Theorem  \ref{th:lowsumcusum} reveals that the $L^{th}$ honest alarm has the same first-order asymptotic performance for any value of $L$ between 1 and $|\cB|$. In the absence of any information regarding the size of the affected subset,  $L$ needs to be set equal to 1. However,  Theorem  \ref{th:lowsumcusum} does not reveal how to select
 $L$  when  the size of the affected subset, $|\cB|$, is known in advance.  This question was addressed in \cite{sourabh_conf}, where it was shown, under the second moment assumption \eqref{second},  that the second-order term in the asymptotic expansion  of the detection delay of the $L^{th}$  alarm is a term of order $\sqrt{\log \gamma}$ whose coefficient is decreasing in $L$. This suggests setting  $L=1$  \textit{independently of  any  prior information regarding the size of the affected subset}. \\
\end{remark}

We  now  establish the asymptotic optimality of the voting rule, $S_L$, when the size of the affected subset is equal to $L$, without the second moment condition that was assumed in \cite{sourabh_conf}.   Moreover, we show that the exponential lower bound \eqref{ARL_voting} is sharp in the exponent. \\

\begin{theorem}\label{th:voting_optimality}
Let  $1 \leq L \leq  |\cN|$. If    $h=h_\gamma$ so that  $ \Exp_\infty[S_{L}(h_\gamma)]=\gamma$, then as $\gamma \rightarrow \infty$
 \begin{equation} \label{ARL_majority_approximation}
h_\gamma \sim   (\log \gamma)/L. 
\end{equation}
 If also  $L \leq |\cB|$, then 
 \begin{equation} \label{ADD_majority_approximation}
  \cJ\left[S_L (h_\gamma)\right] \sim  \frac{\log \gamma}{L \, \cI }.
\end{equation}
When in particular $L=|\cB|$, 
 \begin{equation} \label{majority_optimality}
\cJ\left[S_{|\cB|} (h_\gamma)\right] \sim \frac{\log \gamma}{|\cB| \, \cI } \sim \inf_{T \in \ccg} \cJ[T]. \\
\end{equation}
\end{theorem}

\begin{proof}
 Asymptotic approximation \eqref{ADD_majority_approximation} follows directly
from \eqref{ADD_one_shot} and \eqref{ARL_majority_approximation}, therefore it suffices to show the other two claims of the theorem.

For every $1 \leq L \leq |\cN|$ we have  from  \eqref{ARL_voting} that  as $\gamma \rightarrow \infty$
\begin{align} \label{show0}
 h_\gamma \leq   (\log \gamma)/L +\calo(1).
 \end{align}
  Therefore, setting $h=h_\gamma$ in  \eqref{ADD_one_shot} 
and recalling Lemma \ref{lem:worst_case}  we obtain 
 \begin{equation*}
  \cJ\left[S_L (h_\gamma)\right] \leq \frac{\log \gamma}{L \, \cI } (1+o(1))
\end{equation*}
whenever $1 \leq L \leq |\cB|$.  When in particular $L=|\cB|$, this asymptotic upper bound  coincides with  the optimal asymptotic performance \eqref{opt_perf}, and implies \eqref{majority_optimality}. 

In view of  \eqref{show0}, in order to establish  \eqref{ARL_majority_approximation}  it suffices to show that as $\gamma \rightarrow \infty$
$$
h_\gamma \geq (\log \gamma)/L((1+o(1)).
$$
% as $\gamma \rightarrow \infty$, or equivalently 
%$$
%\liminf_{h \rightarrow \infty} \frac{h_\gamma}{\log \gamma}  \geq  \frac{1}{L}.
%$$
We will prove this by contradiction. Indeed,   suppose there is a subsequence $h'_\gamma$ such that $h'_\gamma \leq  (\log \gamma)/L'$ as $\gamma \rightarrow \infty$ for some  $L'>L$. Then, whenever $L \leq |\cB|$,  from   \eqref{ADD_one_shot} we will have
$$ \cJ\left[S_L (h_\gamma)\right] \leq \frac{\log \gamma}{L' \, \cI } (1+o(1)),
$$
which  contradicts \eqref{majority_optimality} when $L=|\cB|$. \\
\end{proof}

\begin{remark} \label{rem:voting_L}
Suppose that it is known in advance that at least $Q$ sensors are affected, i.e., $|\cB| \geq Q$, where $Q$ is some known number between $1$ and $|\cN|$. Any value of $L$ between $1$ and $Q$ guarantees non-trivial detection delay for the corresponding voting rule,  but the resulting first-order asymptotic detection delay is now decreasing in $L$. This suggests setting $L$ equal to the largest possible value, i.e., $Q$. Note however that  due to the effect of the second-order term in the asymptotic approximation of the detection delay (see Remark \ref{rem:one_shot_choice_L}), it has been argued that a smaller value for $L$, such as  $L=\lceil Q/2 \rceil$ may lead to  better  performance  in practice \cite{sourabh_conf}.  %a conclusion supported by a second-order asymptotic analysis under the second moment condition \eqref{second}  in \cite{sourabh_conf}. %showed that the second-order term in the asymptotic expansion of the detection delay of $S_{L}$ is a term of order $\sqrt{\log \gamma}$,  decreasing in $L$. This implies that  that setting $L=[Q/2]$ may lead to  better, non-asymptotic performance  in practice, a conjecture that was supported numerically as well \cite{sourabh_conf}.  %In any case, it is fair to say that the more precise information is available regarding the size of the affected subset, the more efficient  the family of voting rules is relative to the  first-alarm.
\end{remark}
 
\subsection{Sum-CUSUM} \label{subsec:sum_cusum}
Let $\cC \subseteq \cN$ and denote by $\rho_{\cC}(h)$ the first time 
\textit{the sum of  the  local  CUSUM statistics in $\cC$} is above $h$, i.e., 
\begin{align} \label{sum}
\rho_{\cC}(h) = \inf\left\{t \in \bN: \sum_{k \in \cC} W_t^{k} \geq h\right\}.
 \end{align}  
This detection rule, to which we will refer as  \textit{Sum-CUSUM}, was  proposed in  \cite{mei_bio} and was shown to 
achieve the optimal performance to a first-order asymptotic approximation \eqref{opt_perf} for  any possible affected subset when $\cC=\cN$. In   Theorem \ref{th:sum_cusum_optimality} we characterize  the  first-order asymptotic performance of $\rho_\cC$ whenever $\cC$  intersects with the affected subset, $\cB$, and recover the result in \cite{mei_bio} as a special case. We note however that our  proof differs from that in  \cite{mei_bio} as far as it concerns the proof of the lower bound  in  \eqref{ALB}.   \\

 %althoug  does not require any information regarding the affected subset,  unlike one-shot schemes and voting rules.  Nevertheless, it was shown in  \cite{mei_bio}  to achieve the optimal asymptotic performance \eqref{opt_perf} for   any possible affected subset, $\cB$.  %Note that unlike the optimal CUSUM test, SUM-CUSUM does not require knowledge of the affected subset. Nevertheless, it  preserves the optimal asymptotic performance for any possible affected subset. 

\begin{lemma} \label{lem:sum}
For any $\cC \subseteq \cN$ we have  as $h \rightarrow \infty$
\begin{align} \label{ALB}
\Theta(1) \,  e^{h} \; h^{1-|\cC|}  \leq 
\Exp_\infty \left[ \rho_{\cC}(h) \right] &\leq \Theta(1) \,    e^{h}.
\end{align}
\end{lemma}

\begin{proof}
By definition, $\rho_\cC(h) \leq \sigma_{(1)}(h)$ for every $h>0$, therefore the upper bound in \eqref{ALB} follows from \eqref{ARL_one_shot}. It remains to prove the lower bound.  From \eqref{domin} it follows that $W_s^k$ is stochastically bounded by an exponential random variable with mean 1 for every $s \in \bN$  and $k \in \cC$. As a result, $\sum_{k \in \cC}  W_s^{k}$ is stochastically bounded by an Erlang random variable with parameter $|\cC|$, i.e.,
\begin{align} \label{H}
\Pro_\infty \left( \sum_{k \in \cC}  W_s^{k} \geq h \right) 
 \leq H_{\cC}(h) \equiv  e^{-h} \sum_{j=0}^{|\cC|-1} \frac{h^j}{ j!}.
 \end{align}
For any $t \in \bN$ and $h>0$, from Boole's inequality  we have 
\begin{align} \label{q}
\begin{split}
\Pro_\infty(\rho_{\cC}(h)\leq  t) &= \Pro_\infty\left( \max_{1 \leq s \leq t} \sum_{k \in \cC}  W_s^{k} \geq h\right) \\
&\leq 
\sum_{s=1}^t \Pro_\infty\left( \sum_{k \in \cC}  W_s^{k} \geq h \right) \\
&\leq   t \, H_{\cC}(h),
\end{split}
 \end{align}
and  from Lemma \ref{newlemma}(i) in the Appendix we conclude that   for every $h>0$ we have 
$\Exp_\infty[ \rho_{\cC}(h) ] \geq  1/(2 H_{\cC}(h)).$
 From the definition of  $H_{\cC}$ in \eqref{H} we have   as $h \rightarrow \infty$
\begin{align} \label{qq}
H_{\cC}(h) \sim \frac{e^{-h} \; h^{|\cC|-1}}{(|\cC|-1)!},
 \end{align}
which implies the asymptotic lower bound in \eqref{ALB}. \\
\end{proof}

\begin{lemma} \label{lem:sum1}
If   $\cC \cap \cB \neq \emptyset$, then   as $h \rightarrow \infty$ 
 \begin{equation} \label{old334}
\Exp_{0}^{\cB} \left[ \rho_{\cC}(h) \right] \sim  \frac{h} { |\cC \cap \cB|  \cI}.
 \end{equation}

\end{lemma}

\begin{proof}
For  every $t \in \bN$ we observe that
$$
W_t^{\cB \cap \cC } \leq \sum_{k \in \cB \cap \cC} W_t^k \leq 
\sum_{k \in \cC } W_t^k,
$$
therefore   for every $h>0$ we have
 \begin{align}  \label{path}
 \rho_{\cC}(h) \leq \sigma_{\cC \cap \cB}(h),
 \end{align}   
 and  from the asymptotic approximation \eqref{opt} we obtain 
$$\Exp_{0}^{\cB}[\rho_{\cC}(h)] \leq \Exp_{0}^{\cB}[\sigma_{\cC \cap \cB}(h)] \sim \frac{h}{|\cC \cap B| \, \cI}.
$$
It remains to show that 
$$
\Exp_{0}^{\cB}[\rho_{\cC}(h)] \geq \frac{h}{|\cC \cap B| \, \cI} \; (1+o(1)).
$$
This will  follow directly from Fatou's lemma  as soon as we prove 
that as $h \rightarrow \infty$
 \begin{align}  \label{as_sum}
\frac{\rho_{\cC}(h)}{h}  \overset{\Pro_0^{\cB} } \longrightarrow \frac{1}{|\cC \cap B|  \, \cI}.
\end{align}   
In view of  decomposition \eqref{cusum_decomposition}, 
for every $t \in \bN$ we have 
  $$
 \sum_{k \in \cC } W_t^k= \sum_{k \in \cC \cap \cB } Z_t^k+ \sum_{k \in \cC \cap \cB } m_t^k+ \sum_{k \in \cC \setminus \cB} W_t^k,
$$
and  \eqref{as_sum} will then follow from non-linear renewal theory, see, e.g.,  \cite[Lemma 2.6.1]{TNB_book2014}, if we show that 
$$\frac{1}{t} \, \max_{1 \leq s \leq t}  \sum_{k \in \cC \cap \cB} m_s^k \quad \text{and} \quad \frac{1}{t} \, \max_{1 \leq s \leq t}  \sum_{k \in \cC \setminus \cB} W_s^k
$$
converge to 0 in probability under $\Pro_0^\cB$ as $t \rightarrow \infty$.   The first one holds  because
for every  $k \in \cB$ the random walk $Z^k$ has positive drift $\cI$ (recall \eqref{first}), and as a result  for every  $t \in \bN$ we have  $0\ \leq  -m_t^k \leq -\min_{t \geq 0} Z_t^k<\infty$  almost surely under $\Pro_0^\cB$.  The second one holds because  from \eqref{q} we have for any $\epsilon>0$ that 
$$
\Pro_\infty  \left( \max_{1 \leq s \leq t} \sum_{k \in \cC \setminus \cB} W_s^k >t \epsilon \right) \leq t \, H_{\cC\setminus \cB}(t \epsilon),
$$
and the upper bound goes to 0 as $t \rightarrow \infty$  in view of \eqref{qq}.\\
\end{proof}

We now characterize the performance of  $\rho_C$ up to a first-order asymptotic approximation whenever $\cC$ intersects with $\cB$.\\ % achieves the optimal asymptotic performance whenever $\cB \subseteq \cC$. Clearly, this is the case for every affected subset $\cB$ when we set $\cC=\cN$. 

\begin{theorem} \label{th:sum_cusum_optimality}
Suppose   $\cC \cap \cB \neq \emptyset$. If  $h=h_\gamma$ is so that  $ \Exp_{\infty} \left[ \rho_{\cC}(h_\gamma)\right]=\gamma$, then
as $\gamma \rightarrow \infty$  we have 
 \begin{equation} \label{ARL_sum_approximation}
 h_\gamma  \sim \log \gamma,
\end{equation}
and 
 \begin{equation} \label{old3344}
 \cJ[\rho_{\cC}(h_\gamma)] \sim   \frac{\log \gamma} { |\cC \cap \cB|  \, \cI}.
 \end{equation}
 When in particular   $\cB \subseteq \cC$, we have as $\gamma \rightarrow \infty$
 \begin{equation} \label{old33}
 \cJ[\rho_{\cC}(h_\gamma)]  \sim \frac{\log \gamma} {|\cB|  \,\cI} \sim \inf_{T \in \ccg}\cJ[T] .
\end{equation}
\end{theorem}

\begin{proof}
From  \eqref{ALB} it follows that 
if we set $h=h_\gamma$ so that $\Exp_{\infty} [ \rho_{\cC}(h_\gamma)]=\gamma$, then  as $\gamma \rightarrow \infty$ we have 
$$
\Theta(1) \,e^{h_\gamma} \, h_\gamma^{1-|\cC|} \leq \gamma  \leq  \Theta(1) \, e^{h_\gamma}  \, (1+o(1)).
$$
Taking logarithms and dividing by $\log \gamma$ we obtain \eqref{ARL_sum_approximation}. The asymptotic approximation  in \eqref{old3344} follows from \eqref{old334}
and  \eqref{ARL_sum_approximation}.  Comparing the asymptotic upper bound \eqref{old3344} when $\cB\subseteq \cC$ with the
optimal asymptotic performance  in \eqref{opt_perf} we obtain \eqref{old33}.
\end{proof}

\subsection{Top-Sum-CUSUM}
Let $1 \leq L \leq |\cN|$ and denote by  $\widehat{S}_L(h)$  the first time  the \textit{sum of the $L$ largest} honest local CUSUM statistics is above $h$, i.e., 
\begin{align} \label{hat}
 \widehat{S}_{L}(h) = \inf\left\{t \in \bN: \sum_{k=1}^{L} W_t^{(|\cN|-k+1)} \geq h\right\}, 
 \end{align} 
 to which we  will refer  as  \textit{Top-Sum-CUSUM}. This detection rule  reduces to the \textit{first honest alarm}, $\sigma_{(1)}$,  when $L=1$,  and to  \textit{Sum-CUSUM}, $\rho_\cN$, when  $L=|\cN|$.  It has been  proposed  \cite{mei_sympo} as an efficient modification of  \textit{Sum-CUSUM}  when  the size of the affected subset  is known to be smaller or equal to $L$, i.e., $|\cB| \leq L$. Here, we analyze its asymptotic performance for any value of $L$. \\

\begin{lemma} \label{lem:ADD_top_sum_cusum}
Let $1 \leq L \leq |\cN|$.  Then, for any subset $\cB \subseteq \cN$ we have as $h \rightarrow \infty$
\begin{equation}  \label{hat00}
\Exp_0^{\cB} \left[\widehat{S}_{L}(h) \right] \sim  \frac{h }{\min\{L, |\cB|\} \; \cI }.
\end{equation}
%When in particular  all honest sensors are affected ($\cB=\cN$), then as  $h \rightarrow \infty$ we have 
%\begin{equation}  \label{hat0000}
%\cJN \left[\widehat{S}_{L}(h) \right] \sim  \frac{h}{L  \; \cI }. \\
%\end{equation}
\end{lemma}

\begin{proof}
%From Lemma \ref{lem:worst_case}  we have $\cJ [\widehat{S}_{L}(h) ] =\Exp_0^{\cB}[ \widehat{S}_{L}(h) ]$ for every $h>0$. 
We observe that 
\begin{align} \label{repre}
\begin{split}
  \widehat{S}_{L}(h)   
  %&=  \inf\left\{t\in \bN: \max_{\cC \subseteq \cN: |\cC| = %L}  \, \sum_{k \in \cC} W_t^k \geq h \right\}  \\
  &=  \min_{\cC \subseteq \cN: |\cC| = L}  \rho_{\cC}(h),
  \end{split}
 \end{align}
 When $L> |\cN|-|\cB|$,  $\widehat{S}_{L}(h)$ can be expressed as follows
 $$\phi_{L}^{\cB}(h)=  \min_{\cC \subseteq \cN: \cC \cap \cB \neq \emptyset,  |\cC| = L}  \rho_{\cC}(h).
$$
 When $L\leq |\cN|-|\cB|$,  $\widehat{S}_{L}$  can be represented  as the minimum of two independent stopping times as follows: 
\begin{align} \label{repre2}
  \widehat{S}_{L}(h)   &=  \min\left\{ \phi_{L}^{\cB}(h), \chi_{L}^{\cB}(h)\right\},
  \end{align}    
  where
 $$\chi_{L}^{\cB}(h)=  \min_{\cC \subseteq \cN: \cC \cap \cB = \emptyset,  |\cC| = L}  \rho_{\cC}(h).
$$   
It then suffices to show that   as $h \rightarrow \infty$
\begin{align} \label{step1}
  \Exp_0^{\cB} \left[\phi_{L}^{\cB} \right]  \sim  \frac{h}
  {\min\{L, |\cB|\} \; \cI }    
  \end{align}    
and 
\begin{align} \label{step2}
\Exp_0^\cB \left[  \widehat{S}_{L}(h) \right]  =   \Exp_0^\cB \left[\phi_{L}^{\cB}(h) \right] -o(1).
    \end{align}
We start with the proof of \eqref{step1}.  For every  subset $\cC \subseteq \cN$  of size $L$ that intersects with $\cB$ we have  
\begin{align} \label{repre2}
 \Exp_0^{\cB} \left[ \phi^\cB_{L}(h) \right]  &\leq \Exp_0^{\cB} \left[ \rho_\cC(h) \right]  \sim 
     \frac{h} {|\cC \cap \cB| \, \cI},
 \end{align} 
 where the asymptotic equivalence follows from  \eqref{old334}. Minimizing the asymptotic upper bound with respect to $\cC$ we obtain  
 \begin{align*} 
   \Exp_0^{\cB} \left[ \phi^\cB_{L}(h) \right]    &\leq 
 %\min_{\cC \subseteq \cN :\cC \cap \cB \neq \emptyset,|\cC| = %L}   \frac{h} {|\cC \cap \cB| \, \cI} \; (1+o(1)) \\
  \frac{h}{\min\{L, |\cB|\} \; \cI } \; (1+o(1)),
 \end{align*} 
 since 
 \begin{align}  \label{bbb}
 \max_{\cC \subseteq \cN :  \cC \cap \cB \neq \emptyset, |\cC| = L} |\cC \cap \cB| = \min\{L, |\cB|\}.
 \end{align}
% Indeed, $|\cC \cap \cB| \leq \max\{|\cC|, |\cB|}$ and the upper bound is equal to $L$ when 
In order to prove \eqref{step1}, it remains to show that
 \begin{align*}
   \Exp_0^{\cB} \left[ \phi^\cB_{L}(h) \right]    &\geq   \frac{h}{\min\{L, |\cB|\} \; \cI } \; (1+o(1)).
 \end{align*} 
This follows   from Fatou's lemma  and 
$$ \frac{\phi_{L}^{\cB}(h)}{h} \overset{\Pro_0^{\cB}} \longrightarrow  \frac{1}{\min\{L, |\cB|\} \cI} ,
$$
the latter being a consequence of  \eqref{as_sum}  and \eqref{bbb}.

In order to prove \eqref{step2}, we focus without loss of generality on the case that $L\leq |\cN|-|\cB|$,  and we utilize  Lemma \ref{newlemma}(ii) in the Appendix. We observe that for every 
$\cC \subseteq \cN$ of size $L$ that intersects with  $\cB$,  we have 
$\phi_{L}^{\cB}(h) \leq   \rho_{\cC}(h) \leq \sigma_{\cC \cap \cB} (h)
$
for every $h$, where the second inequality follows from \eqref{path}. Therefore, setting $r=2$ in \eqref{cusum_moment_approximation} we obtain 
$$   \Exp_0^{\cB} \left[  \left(\phi^\cB_{L}(h) \right)^2 \right]   =\calo(h^2).
$$
Moreover,  for every $h>0$ we have  $\chi_{L}^{\cB}(h) \geq  \rho_{\cN \setminus\cB }(h)$,  and consequently for every $t \in \bN$ we obtain 
 \begin{align*}
\Pro_\infty \left( \chi_{L}^{\cB}(h) \leq t\right) &\leq \Pro_\infty \left( \rho_{\cN \setminus\cB }(h) \leq t\right)  \\
&\leq t \, H_{\cN\setminus \cC}(h),
 \end{align*} 
where the second inequality follows from \eqref{q}.  From the definition of $H_{\cN\setminus \cC}$ in \eqref{H} it follows that $h^2 H_{\cN\setminus \cC}(h)  \rightarrow 0$ as $h \rightarrow \infty$, which completes the proof of \eqref{step2}.  

\end{proof}

We now the  characterize, the  asymptotic performance of  \textit{Top-Sum-CUSUM}.

\begin{theorem}
Let $1 \leq L \leq |\cN|$. If    $h=h_\gamma$ is selected so that  $ \Exp_{\infty} [ \widehat{S}_{L} (h_\gamma)]=\gamma$,  then  as $\gamma \rightarrow \infty$ we have 
$h_\gamma \sim \log \gamma$  and 
\begin{equation}  \label{hat10}
\cJ \left[\widehat{S}_{L}(h_\gamma) \right] \sim  \frac{\log \gamma }{  \min\{|\cB|, L\} \,\cI }.
\end{equation}
When in particular  $L \geq |\cB|$, 
\begin{equation}  \label{hat11}
\cJ \left[\widehat{S}_{L}(h_\gamma) \right] \sim   \frac{\log \gamma}{  |\cB| \,\cI } \sim \inf_{T \in \ccg} \cJ[T].
\end{equation}

\end{theorem}

\begin{proof}
For every $h>0$ and $1 \leq L \leq |\cN|$ it is clear that 
$$\rho_{\cN}(h) \equiv  \widehat{S}_{|\cN|}(h) \leq  \widehat{S}_{L}(h)  \leq \widehat{S}_{1}(h) \equiv \sigma_{(1)}(h).
$$
From  \eqref{ARL_one_shot_approximation} and \eqref{ARL_sum_approximation} it follows that 
if we set $h=h_\gamma$ so that $\Exp_{\infty} [ \widehat{S}_{L} (h_\gamma)]=\gamma$, then  as $\gamma \rightarrow \infty$ we have 
$$
\Theta(1) \,e^{h_\gamma} \, h_\gamma^{1-|\cN|} \leq \gamma  \leq  \Theta(1) \, e^{h_\gamma}  \, (1+o(1)).
$$
Taking logarithms and dividing by $h_\gamma$ we obtain that $h_\gamma \sim \log \gamma$. This observation, combined with Lemma 
\ref{lem:ADD_top_sum_cusum}   implies \eqref{hat10}. Finally, \eqref{hat11} follows by comparing  \eqref{hat10} with the   optimal  asymptotic performance   \eqref{opt_perf} when $L\geq |\cB|$.
\end{proof}

\subsection{Low-Sum-CUSUM} \label{subsec:LowSumCUSUM}
Let $1 \leq L \leq |\cN|$ and denote by  $\widetilde{S}_L(h)$  the first time  the \textit{sum of the $L$ smallest honest CUSUM statistics} is above $h$, i.e.,  
\begin{align} \label{tilde}
 \widetilde{S}_{L}(h) = \inf\left\{t \in \bN: \sum_{k=1}^{L} W_t^{(k)} \geq h\right\},
 \end{align} 
 to which we will refer as  \textit{Low-Sum-CUSUM}. It is clear that $\widetilde{S}_{L}$ reduces to  the \textit{consensus rule}, $S_{|\cN|}$, when $L=1$, and  to  \textit{Sum-CUSUM}, $\rho_{\cN}$, when $L=|\cN|$. To the best of our knowledge, this procedure  has not been studied  when $1 < L < |\cN|$. In Theorem \ref{th:lowsumcusum} we  show that $\widetilde{S}_{L}$ is asymptotically optimal, \textit{for any choice of $L$}, when all honest sensors are affected by the change ($\cB=\cN$). This rule will turn out to be  well-suited to address the presence of corrupt sensors, which is our focus in the next section.\\

\begin{lemma} \label{ARL_low_sum_cusum}
Let $1 \leq L \leq |\cN|$.  As $h \rightarrow \infty$ we have
\begin{align}\label{ARL_lowsumcusum}
 \Exp_\infty \left[ \widetilde{S}_{L}(h) \right] &\geq \Theta(1) \,  \exp\{ (|\cN|  / L) h\}.
\end{align} 
\end{lemma}

\begin{proof}
For any $t \in \bN$ and $h>0$, from Boole's inequality  we have 
\begin{align*}
\Pro_\infty(\widetilde{S}_{L}(h) \leq  t) &= \Pro_\infty\left( \max_{1 \leq s \leq t} \sum_{k=1}^{L} W_s^{(k)} \geq h\right) \\
&\leq  \sum_{s=1}^t \Pro_\infty\left( \sum_{k=1}^{L} W_s^{(k)}  \geq h \right).
 \end{align*}
 From \eqref{domin} it follows that $W_s^{(k)}$ is stochastically bounded by an exponential random variable with mean 1, therefore $\sum_{k=1}^{L} W_s^{(k)} $ is stochastically bounded by the sum of the smallest $L$, among $|\cN|$, independent exponential random variables with mean 1. Therefore, from 
 Lemma \ref{lem:asy} in the Appendix we obtain 
\begin{align*}
\Pro_\infty\left( \sum_{k=1}^{L} W_s^{(k)}  \geq h \right) \leq  G_{L}(h)
 \end{align*}
 where $G_L$ is defined in \eqref{GG}. From
Lemma \ref{newlemma}(i) in the Appendix we  conclude that  $\Exp_\infty[ \widetilde{S}_{L}(h) ] \geq  1/(2 G_{L}(h))$ for every $h>0$, and  from Lemma \ref{lem:asy} in the Appendix we have   that as $h \rightarrow \infty$ 
\begin{align*}
  G_{L}(h) &\sim \Theta(1) \,  \exp\{ -(|\cN|/L) h \} ,
\end{align*}
which implies \eqref{ARL_lowsumcusum}. \\
\end{proof}

It is clear that  \textit{Low-Sum-CUSUM} has  non-trivial detection performance whenever  $L$ is larger than the number of non-affected sensors, $|\cN| -|\cB|$, so that there is always at least one term in the detection statistic that corresponds to a  local CUSUM statistic from an affected sensor. Under this assumption, we now characterize  its worst-case detection delay up to a first-order approximation. \\

\begin{lemma}  \label{ADD_lowsumcusum}
If $L > |\cN| -|\cB|$, then   as $h \rightarrow \infty$  we have
\begin{align}  \label{upper6}
\Exp_0^\cB \left[\widetilde{S}_{L}(h) \right] \sim \frac{h}{ 
(L-|\cN|+|\cB|) \; \cI}.
\end{align} 
\end{lemma}

\begin{proof}
For any $t \in \bN$ we have 
\begin{align} \label{rep}
\sum_{k=1}^{L} W_t^{(k)}  &= \min_{\cC \subseteq \cN: |\cC| = L}  \, \sum_{k \in \cC} W_t^k.
\end{align}
When $L >|\cN| -|\cB|$,  a lower bound for the sum of the $L$ smallest honest CUSUM statistics is obtained when we set the CUSUM statistics in the $|\cN|-|\cB|$ unaffected sensors equal to 0. Then,  in view of \eqref{rep}, we have  for every $t \in \bN$ 
\begin{align*}
\sum_{k=1}^{L} W_t^{(k)} &\geq  \min_{\cC \subseteq \cB: |\cC| = L-|\cN|+|\cB|}  \, \sum_{k \in \cC} W_t^k  \\
&  \geq  \min_{\cC \subseteq \cB: |\cC| = L-|\cN|+|\cB|}  \, \sum_{k \in \cC} Z_t^k,
\end{align*}
where the second inequality holds because $Z_t^k \leq W_t^k$ for every $t$. As a result, for every $h>0$ we have 
\begin{align*}   
  \widetilde{S}_{L}(h)   &\leq  \inf\left\{t\in \bN: \min_{\cC \subseteq \cB: |\cC| = L-|\cN|+|\cB|}  \, \sum_{k \in \cC} Z_t^k \geq h \right\}.
 \end{align*}
%i.e.,   $\widetilde{S}_{L}(h)$ is bounded by the first time that 
%$ L-|\cN|+|\cB|$ random walks with drift $\cI$ are simultaneously above $h$. 
Then, from Lemma \ref{lem:simultaneous} in the Appendix it follows that 
  as $h \rightarrow \infty$ 
\begin{equation*} 
\Exp_0^{\cB} \left[\widetilde{S}_{L}(h) \right] \leq \frac{h}{ 
(L-|\cN|+|\cB|) \, \cI} \; (1+o(1)). 
\end{equation*}
In order to show that this asymptotic upper bound is sharp, we observe that from the definition of $\rho_\cC$ in \eqref{sum} and \eqref{rep} we have 
\begin{align}  \label{ineq}
\begin{split}
\widetilde{S}_{L}(h) &\geq \max_{\cC \subseteq \cN: |\cC| = L} \rho_{\cC}(h) \\
&\geq \max_{\cC \subseteq \cN: \cC \cap \cB \neq \emptyset, |\cC| = L} \rho_{\cC}(h) \equiv \psi_{L}^{\cB}(h).
\end{split}
\end{align}
From \eqref{as_sum} we know that for any subset $\cC$ that intersects with $\cB$ 
$$
 \frac{\rho_\cC(h)}{h} \overset{\Pro_0^{\cB}} \longrightarrow  \frac{1}{|\cC \cap \cB|\, \cI }, 
$$
and this implies 
$$
 \frac{\psi_{L}^{\cB}(h)}{h} \overset{\Pro_0^{\cB}} \longrightarrow \frac{1}{(L-|\cN|+|\cB|) \, \cI },
$$
since 
$$
\min_{\cC \subseteq \cN: \cC \cap \cB \neq \emptyset,|\cC| = L}
|\cC \cap \cB|=  L-|\cN|+|\cB|.
$$
Therefore, from Fatou's lemma   we obtain 
 \begin{align*} \Exp_0^{\cB} \left[ \psi^\cB_{L}(h) \right]  &\geq   \frac{h}{(L-|\cN|+|\cB|) \; \cI } \; (1+o(1)).
 \end{align*} 
 This asymptotic lower bound and  \eqref{ineq} imply 
 \begin{align*}
\Exp_0^{\cB} \left[\widetilde{S}_{L}(h) \right]  &\geq     \frac{h}{(L-|\cN|+|\cB|) \; \cI } \; (1+o(1)),
 \end{align*} 
which completes the proof. \\
\end{proof}

In the following theorem we show that   \textit{Low-Sum-CUSUM} preserves, \textit{for any choice of $L$}, the asymptotic optimality of the consensus rule $\widetilde{S}_{1}\equiv S_{|\cN|}$, when all honest sensors are affected by the change ($\cB=\cN)$. The asymptotic optimality property  allows us further to show that the exponential lower bound in \eqref{ARL_lowsumcusum} is sharp in the exponent.
% which is the content of Corollary \ref{coro4}. \\
%\end{remark}

\begin{theorem} \label{th:lowsumcusum}
Let $1 \leq L \leq |\cN|$  and  $h=h_\gamma$  so that  $ \Exp_{\infty} [ \widetilde{S}_{L} (h_{\gamma})]=\gamma$. Then,  
as $\gamma \rightarrow \infty$ we have 
\begin{equation}  \label{threshold_lowsum_approximation}
h_\gamma \sim (L/ |\cN|) \, \log \gamma.
\end{equation}
 Suppose  further that  $L > |\cN|-|\cB|$.  Then as $\gamma \rightarrow \infty$
\begin{align}  \label{perf_lowsumcusum}
\cJ \left[\widetilde{S}_{L}(h_\gamma) \right]  \sim
 \frac{L}{|\cN|} \; \frac{\log \gamma}{ (L-|\cN|+|\cB|) \, \cI}.
\end{align} 
In the special case that all honest sensors are affected ($\cB=\cN$),  for every $1 \leq L \leq  |\cN|$ we have
\begin{equation}  \label{optimality_lowsum_cusum}
\cJN \left[\widetilde{S}_{L}(h_\gamma) \right] \sim  \frac{\log \gamma}{  |\cN|\cI}  \sim \inf_{T \in \ccg}\cJN[T].
\end{equation}
\end{theorem}

\begin{proof}
Asymptotic approximation  \eqref{perf_lowsumcusum} follows directly by \eqref{upper6} and  \eqref{threshold_lowsum_approximation}, thus, we focus on the proof of the two other claims.  From  \eqref{ARL_lowsumcusum}  it follows that 
if $h=h_\gamma$  so that  $ \Exp_{\infty} [ \widetilde{S}_{L} (h_{\gamma})]=\gamma$, then as $\gamma \rightarrow \infty$
\begin{align}  \label{ub}
h_\gamma \leq (L/ |\cN|) \, \log \gamma +\calo(1).
\end{align}
 From Lemma \ref{lem:worst_case},  \eqref{upper6} and \eqref{ub}  we  have
 \begin{align*} 
\cJ \left[\widetilde{S}_{L}(h_\gamma) \right]  \leq
 \frac{L}{|\cN|} \; \frac{\log \gamma}{ (L-|\cN|+|\cB|) \, \cI} \;  (1+o(1)). 
\end{align*} 
Comparing this asymptotic upper bound with the optimal asymptotic performance \eqref{opt_perf} when $\cB=\cN$, we obtain the asymptotic optimality property \eqref{optimality_lowsum_cusum}.

In view of \eqref{ub}, in order to prove \eqref{threshold_lowsum_approximation} it suffices to show that   as  $ \gamma \rightarrow \infty$
$$
h_\gamma \geq (L/ |\cN|) \, (\log \gamma) \,   (1+o(1)).
$$  This will follow by arguing via contradiction. 
Indeed, if there is a subsequence $h'_\gamma$ so that $h'_\gamma \leq (L'/ |\cN|) \, \log \gamma $ as $\gamma \rightarrow \infty$ for some $L'>L$, then from Lemma  \ref{ADD_lowsumcusum} it follows that 
$$
\cJN[ \widetilde{S}_L (h_\gamma)] \leq \frac{L'}{L |\cN|}  \;(1+o(1)),
$$
which contradicts \eqref{optimality_lowsum_cusum}.
\end{proof}

%\end{proof}\begin{remark}From  \eqref{perf_lowsumcusum} it follows that the first-order asymptotic performance of the \textit{Low-Sum-CUSUM}  is decreasing in $L$, and setting $L$ equal to its largest possible value, $L=|\cN|$, we recover  the asymptotic optimality of the \textit{Sum-CUSUM}, $\rho_\cN$. \\\end{remark}
%\begin{corollary}\label{coro4}\end{corollary}\begin{proof}

\section{The Byzantine setup} \label{sec:robust} 

In this section we focus on the main theme of this paper, that is the design of multichannel, sequential  change-detection  procedures that are robust in the presence of  corrupt sensors.
% as long as the number of affected sensors is larger than $M$. 
% As we mentioned earlier, our focus is on the worst-case scenario that the number of corrupt sensors is the maximum possible, therefore we will assume from now on that $ K-  |\cN|=M \geq 1$. Moreover, we assume that the  number of affected sensors  is larger than the number of corrupt sensors, i.e.,   $|\cB|>M$. 

We  will still utilize the notation  introduced in Subsection \ref{subsec:classical_notation} when we refer  to events that depend only on  honest sensors. However, we will now need some additional notation when  we do not know whether the  sensors to which we refer are honest or corrupt.  Thus, for any  subset of sensors $\cC \subseteq [K]$  we denote  by $Y^\cC$  the CUSUM statistic for detecting a change in subset $\cC$, i.e., 
\begin{align} \label{YB}
  Y_t^\cC= \left( Y_{t-1}^\cC + \sum_{k \in \cC} \ell_t^k \right)^{+} , \quad t \in \bN, 
\end{align} 
where  $Y_0^\cC=0$ and $\ell_t^k$ is defined in \eqref{llr}. We denote by $\tau_{\cC}(h)$ be the first time the process $Y^\cC$  exceeds a positive threshold $h$, i.e.,  
\begin{align} \label{tB}
 \tau_{\cC}(h) = \inf\left\{t \in \bN: Y_t^{\cC} \geq h\right\}.
 \end{align} 
When $\cC=\{k\}$ for some $k \in [K]$, we simply write 
$Y_t^k$ and $\tau_{k}(h)$, instead of $ Y_t^{\{k\}}$ and  $\tau_{\{k\}}(h)$. Finally,  we use  the following notation for the ordered local CUSUM stopping times and statistics: 
\begin{align} \label{ordered_CUSUMS_byz}
& \tau_{(1)}(h) \leq \ldots \leq \tau_{(K)}(h),  \quad  Y_t^{(1)} \leq \ldots \leq Y_t^{(K)}.
\end{align} 
Our goal  is to design  procedures that  are able to detect the change quickly and reliably in the worst case scenario regarding the corrupt sensors, for any subset of honest sensors that perceive the change. 
To this end, we  assume that there is a user-specified upper bound, $M$, on the number of corrupt sensors,
and we focus our analysis on the worst possible case that there are exactly $M$ corrupt sensors, with the understanding that  the proposed procedures will still be able to detect the change reliably  when the actual number of corrupt sensors is smaller than $M$. Thus, from now on we have $|\cN|=K-M$, and consequently $1 \leq |\cB| \leq K-M$, since  $\cB \subseteq \cN$.  % We will see however that the proposed procedures will not be able to detect the change for every $\cB$, but only when  $|\cB|$ is large enough.

\subsection{The proposed procedures}

We will study three of the  families of multichannel detection schemes that we considered in the previous section: the $L^{th}$ alarm, $\tau_{(L)}(h)$, defined in \eqref{ordered_CUSUMS_byz},  the voting rule  
 \begin{equation} \label{general_voting}
 T_{L}(h)=\inf\left\{t \in \bN: Y_t^{(K-L+1)}  \geq h \right\},
 \end{equation} 
and  \textit{Low-Sum-CUSUM},
\begin{equation} \label{general_low_sum_cusum}
 \widetilde{T}_{L}(h)  =\inf\left\{t \in \bN: \sum_{k=1}^{L} Y_t^{(k)}  \geq h \right\},
 \end{equation} 
 where $L$ is some  number between $1$ and $K$.  Thus, $\tau_{(L)}(h)$ is the first time $L$ sensors, honest or not,  cross threshold $h$,  $T_{L}(h)$ is the first time $L$ CUSUM statistics, honest or not, are simultaneously above threshold $h$, and   $\widetilde{T}_{L}(h)$ is the first time  the sum of the $L$ smallest local,  honest or not,  CUSUM statistics is  larger than $h$.

% where $L$ will be assumed to be between larger than the number of corrupt sensors, $M$, and at most equal to the number of honest sensors.  

% It is important to stress that the three procedures under consideration require different kinds of information and can be implemented with different levels of communication activity from the sensors to the fusion center. In particular,  $\tau_{(M+1)}$ and $\widetilde{T}_{K-M}$ require knowledge of only $M$, the upper bound on the number of corrupt sensors, whereas  $T_{|\cB|}$ requires knowledge of the exact number of affected sensors. Moreover,   while the LowSumCUSUM requires transmission of the full value of the local CUSUM statistic from each sensor at each time, the voting rule requires from each sensor transmission of one bit at each time (declaring whether its local statistic is above or below $h$, whereas  the $(M+1)^{th}$ alarm strategy requires from each sensor to communicate only once once,as soon as its local CUSUM statistic exceeds $h$.

%This scheme assumes knowledge of $M$, the number of corrupt sensors, but does not assume any knowledge of the size of the affected subset, $|\cB|$. 
 
  %Note also that $T_{|\cB|}$ can be implemented in a decentralized fashion.  
 
 %On the other hand, when only  non-trivial Similarly to  the $M+1$ alarm, this detection rule  require knowledge of the number of corrupt sensors, $M$, but not the number of  affected sensors, $|\cB|$. However, unlike the  $M+1$ alarm, $|\cN|$-low-Sum-CUSUM can only be implemented in a centralized fashion. 
 
 \begin{remark}
 We do not consider \textit{Top-Sum-CUSUM} in this context, because with any rule of this form the adversary can  trigger unilaterally false alarms before the change, violating the desired false alarm control. 
 \end{remark}
 
\subsection{Preliminary results}

The following lemma is  important for the subsequent development, as it represents the operating characteristics of the proposed procedures in terms of  operating characteristics of schemes that involve only honest sensors, thus, allowing  us to use the results from the previous section. It also reveals that only values of $L$ larger than $M$ are relevant for all three schemes. \\   %Moreover,  it leads to  upper/lower bounds for the parameter $L$ for the proposed rules to be able to detect the change. 

\begin{lemma} \label{non-asy}
Suppose that $L> M$. Then,  for every threshold $h>0$ and  subset $\cB \subseteq \cN$ we have 
  \begin{align} \label{repre_oneshot}
  \begin{split}
      \cA \left[ \tau_{(L)} (h)\right] &=  \Exp_\infty\left[ \sigma_{(L-M)}(h)\right],  \quad \\
       \cJ \left[\tau_{(L)}(h)\right] &=  
 \Exp_{0}^{\cB} \left[ \sigma_{(L)}(h)\right] ,
\end{split}
 \end{align}

   \begin{align}    \label{repre_voting}
     \begin{split}
      \cA\left[ T_{L}(h)\right] &=\Exp_\infty[S_{L-M}(h)], \\ \cJ\left[T_{L}(h)\right]  &=\Exp_{0}^{\cB} \left[ S_{L}(h)\right],  \\
      \end{split}
      \end{align}
   and 
  \begin{align}  \label{repre_lowsumcusum}
    \begin{split}
 \cA \left[ \widetilde{T}_{L}(h)\right] &=\Exp_\infty \left[\widetilde{S}_{L}(h)  \right], \\
 \quad\cJ\left[\widetilde{T}_{L}(h)\right] & =\Exp_{0}^{\cB} \left[\widetilde{S}_{L-M}(h)\right].
 \end{split}
\end{align}

\end{lemma}

\begin{proof}
For simplicity, we suppress the dependence on  threshold $h$.
As far as it concerns the false alarm rate of the proposed schemes, the  worst-case scenario regarding the data in the  corrupt sensors is when  the CUSUM statistics from the corrupt sensors are never smaller than the ones from the honest sensors, i.e.,
\begin{equation*} 
\min_{k \notin \cN} Y_t^k \geq \max_{k \in \cN} Y_t^k, \quad \forall \; t \in \bN.
\end{equation*}
%in which case all corrupt sensors raise an alarm before any honest sensor does so, i.e.,$\max_{k \notin \cN} \sigma_k \leq \min_{k \in \cN} \sigma_k$.
In this case, $\tau_{(L)}$ coincides with the $(L-M)$- \textit{honest} alarm, i.e., $\tau_{(L)}=\sigma_{(L-M)}$,  
$T_{L}$  stops as soon as $L-M$ of the \textit{honest} CUSUM statistics are  simultaneously  above $h$, i.e., $T_{L} = S_{L-M}$, and  $\widetilde{T}_{L}$  stops when the sum of the lowest $L$ \textit{honest} CUSUM statistics is  above $h$, i.e.,  $\widetilde{T}_{L}= \widetilde{S}_{L}$. 

 %Since $S_{K-2M}$ stops when $K-2M$  honest CUSUM statistics are \textit{simultaneously} above $h$, and $\sigma_{(K-2M)}$ as soon as there have been $K-2M$ alarms in the honest sensors, we have  $S_{K-2M} \geq \sigma_{(K-2M)}$. Moreover,  $\widetilde{S}_{K-M}$ stops as soon as  the sum of all honest CUSUM statistics exceeds $h$ and $\sigma_{(1)}$ as soon as  the  maximum honest CUSUM statistic exceeds $h$, thus, This implies $$ \Exp_\infty[\widetilde{S}_{K-M}] \leq \Exp_\infty[\sigma_{(1)}] \leq \Exp_\infty[\sigma_{(K-2M)}] \leq \Exp_\infty[S_{K-2M}],$$ and consequently \eqref{ineq}.

As far as it concerns  the detection delay of the proposed rules, the  worst case scenario regarding the data in the corrupt sensors is when
$$Y_t^k=0, \quad \forall \; k \notin \cN, \quad \forall \; t \in \bN.
$$
% and consequently $\tau_k=\infty$ for every $k \notin \cN$.
Then, $\tau_{(L)}$ stops  as soon as  $L$ \textit{honest} sensors have raised an alarm, i.e., $\tau_{(L)}=\sigma_{(L)}$, $T_{L}$ stops at the first time  $L$ \textit{honest} CUSUM statistics are simultaneously above $h$, i.e.,  $T_{L}= S_{L}$,  whereas  $\widetilde{T}_{L}$ stops  when the sum of the $L-M$ lowest \textit{honest} CUSUM statistics crosses $h$,  i.e.,  $\widetilde{T}_{L}= \widetilde{S}_{L-M}$. In view of Lemma \ref{lem:worst_case},  this completes the proof.

%for the history from the honest sensors  before the change is that  $W_\tau^k=0$ for every $k \in \cN$.  Moreover, since the vector process $(W^k, k \in \cN)$ restarts afresh when all local CUSUMs hit 0 simultaneously, 
%Finally,  in order to prove \eqref{ineq2}, it suffices to observe that  $\sigma_{(M+1)} \leq  \sigma_{(|\cB|)}   \leq  S_{|\cB|}$. % \geq S_{K-2M} \geq \widetilde{S}_{K-2M}$.
\end{proof}

\subsection{The range of L}

From \eqref{repre_oneshot}-\eqref{repre_voting}  it follows that for   $\tau_{(L)}$  and $T_{L}$  to control the worst-case false alarm rate, $L$ needs to be larger than $M$, and for   $\tau_{(L)}$  and $T_{L}$ to have non-trivial detection delay, $L$ needs to be at most equal to the size of the affected subset, $|\cB|$,  or equivalently   at least $L$ honest sensors need to be affected by the change.  Thus,  for these detection rules we will require that 
\begin{equation} \label{always1}
M+1 \leq L \leq |\cB| \leq K-M,
\end{equation}
where the last inequality always holds because $|\cB| \leq |\cN| = K-M$. 

 On the other hand, from  \eqref{repre_lowsumcusum} it follows that  for  \textit{Low-Sum-CUSUM}, $\widetilde{T}_{L}$,  to control the worst case false alarm rate, $L$ needs to be at most equal to the number of honest sensors, $K-M$ (otherwise, a  corrupt CUSUM statistic will always be included in the detection statistic), and
 for  $\widetilde{T}_{L}$  to have non-trivial detection delay,  we need  not only $L>M$, but also that  $L-M > |\cN|- |\cB|$, or equivalently $L > K-|\cB|$ (see  the  discussion prior to Lemma \ref{ADD_lowsumcusum}). Consequently, for  $\widetilde{T}_{L}$ we  require that
\begin{equation} \label{always2}
K+1 - |\cB| \leq  L \leq   K-M.
\end{equation}
This condition implies that for $\widetilde{T}_{L}$ to detect the change, the size  of the affected subset must satisfy  
$$
K-M \geq|\cB| \geq K+1-L \geq M+1,
$$
where again the first inequality  always holds.

%\begin{remark}
%When we set  $L$ equal to its largest possible value, $|\cN|\equiv K-M$, then \eqref{always2} reduces to \eqref{always1}. We will see that this will turn out to the most  statistically efficient choice for $L$.
%\end{remark}

 \subsection{A special case}
 In view of conditions \eqref{always1}-\eqref{always2},  we focus on the case that   $M+1 \leq |\cB|  \leq |\cN|= K-M$.  A particular case of interest is when  $K-M =M+1$, in which case all honest sensors are affected $(\cB=\cN)$.  Then, conditions  \eqref{always1}-\eqref{always2} imply that  the only possible value  for $L$ is  $M+1$ for all three schemes, and Theorem \ref{non-asy} reveals  a clear ordering for these schemes. This is the  content of Corollary \ref{corol}  below, for which we use the notion of domination in Definition \ref{def1}. \\

\begin{corollary} \label{corol}
Suppose that all honest sensors are affected $(\cB=\cN)$ and that  $|\cN|\equiv K-M = M+1$. Then, the $(M+1)$-alarm,  $\tau_{(M+1)}$, dominates the voting rule, $T_{M+1}$, and the latter   dominates  \textit{Low-Sum-CUSUM}, $\widetilde{T}_{M+1}$. \\
\end{corollary}

\begin{proof}
Fix some arbitrary $h>0$.  Then,  it suffices to show
  \begin{align*} 
  \cA \left[ \widetilde{T}_{M+1} (h)\right] &\leq   \cA\left[T_{M+1}(h) \right]= \cA\left[ \tau_{(M+1)}(h) \right], \\
\cJN \left[\widetilde{T}_{M+1}(h) \right] &=\cJN \left[T_{M+1}(h) \right] \geq  \cJN\left[\tau_{(M+1)} (h)\right].
\end{align*}
 In view of Lemma  \ref{non-asy},  it suffices to show 
  \begin{align*} 
 \Exp_\infty \left[ \widetilde{S}_{M+1}(h) \right] &\leq   \Exp_\infty\left[S_{1}(h) \right]= \Exp_\infty\left[ \sigma_{(1)}(h) \right]. \\
\Exp_0^\cN\left[\widetilde{S}_{1}(h) \right] &=\Exp_0^\cN \left[S_{M+1}(h) \right] \geq \Exp_0^\cN \left[\sigma_{(M+1)}(h) \right].
\end{align*}
By definition,  $\sigma_{(L)}(h)  \leq S_{L}(h)$, with equality when $L=1$, therefore  it suffices to show that 
 \begin{align*} 
\widetilde{S}_{M+1}(h) &\leq  \sigma_{(1)}(h) \quad \text{and} \quad 
\widetilde{S}_{1}(h)  =S_{M+1}(h).
\end{align*}
The equality holds because  by definition  $\widetilde{S}_{1}(h)$  coincides with the consensus rule $S_{|\cN|}(h)$, and we further assume  that $|\cN|=M+1$.  The inequality holds because when $M+1=K-M$,   $\widetilde{S}_{M+1}(h)$ coincides with \textit{Sum-CUSUM}, $\rho_{\cN}(h)$, defined in \eqref{sum}, which can never be larger than the corresponding first  honest alarm, i.e., $\rho_{\cN}(h) \leq \sigma_{(1)}(h)$. \\
%\begin{align*}
%\widetilde{S}_{M+1}(h)   &= \inf\left\{t \in \bN: \sum_{k \in \cN} W_t^{k} \geq h\right\} , \\
% &\leq \inf\left\{t \in \bN: \max_{k \in \cN} W_t^{k} \geq h\right\} = \sigma_{(1)} (h).
%\end{align*}
\end{proof}

In  what follows, we focus on the asymptotic performance of the proposed rules. This will allow us to compare them when   $M+1 <K-M$, and also to provide a quantification of  the inflicted performance loss due to the presence of corrupt sensors  when the false alarm rate is  small.

 \subsection{Asymptotic analysis of the $L^{th}$ alarm} \label{sec:asymptotics}
 
In this section we characterize, to a first-order  asymptotic approximation, the  performance of the $L^{th}$  alarm, $\tau_{(L)}$, when \eqref{always1} holds.  \\

\begin{theorem} \label{theo:alarm}
Suppose  $L > M$.    If  $h=h_\gamma$  so that 
$\Exp_\infty\left[ \sigma_{(L-M)} (h_\gamma) \right]=\gamma$, then $\tau_{(L)} (h_\gamma) \in \ccg$. If, additionally, $|\cB| \geq L$, then    as $\gamma \rightarrow \infty$  
\begin{align} \label{alarm}
\cJ\left[\tau_{(L)} (h_\gamma)\right] \sim ( \log \gamma ) / \cI .
\end{align}
\end{theorem}

\begin{proof}
From  \eqref{repre_oneshot}  we have  that for every $h>0$ 
\begin{align*}
\cJ \left[\tau_{(L)}(h) \right] &=\Exp_{0}^{\cB} \left[\sigma_{(L)}(h) \right],  \\\
\cA \left[\tau_{(L)}(h) \right] &= \Exp_{\infty} \left[\sigma_{(L-M)}(h)\right].
\end{align*}
Thus, it suffices to show that   if  $h=h_\gamma$  so that 
$\Exp_\infty[ \sigma_{(L-M)} (h_\gamma)]=\gamma$,
then 
\begin{equation} \label{j}
\Exp_{0}^{\cB} \left[\sigma_{(L)}(h_\gamma) \right] \sim  (\log \gamma)/ \cI.
\end{equation} 
From \eqref{ADD_one_shot}  we have $\Exp_{0}^{\cB} [\sigma_{(L)}(h) ] \sim h/\cI$ as $h \rightarrow \infty$, and from 
\eqref{ARL_one_shot_approximation} that  $h_\gamma \sim \log \gamma$ as $\gamma \rightarrow \infty$, which implies \eqref{j}. \\ 
%\end{equation}
%Finally, from  Lemma \ref{lem0} it follows that  when  $\cB=\cN$,  the asymptotic upper bounds in \eqref{jjj}, and consequently \eqref{jjjj},  become  asymptotic equivalences, which implies \eqref{alarm0}.\\
%and for any $\cB \subseteq \cN$ if additionally \eqref{second} holds. This completes the proof. 
\end{proof}

% Therefore,  $\tau_{(M+1)}$ is  dominated by the first honest alarm, $\sigma_{(1)}$. For the latter detection rule, it is known from \eqref{old21}  that if  $h=h_\gamma$ is selected such that $\Exp_\infty\left[ \sigma_{(1)} (h_\gamma)\right]=\gamma$, then  \begin{equation} \label{in2}\cJ[\sigma_{(1)}(h_\gamma)]  \sim \frac{\log \gamma}{\cI}. \end{equation} Therefore, we conclude that \begin{align*}  \liminf_{\gamma \rightarrow \infty} \frac{\cJ[\tau_{(M+1)} (h_\gamma)]}{\log \gamma} \geq \frac{1}{\cI }, \end{align*} and it remains to show that 
%\begin{align} \label{alarm}\limsup_{\gamma \rightarrow \infty} \frac{\cJ[\tau_{(M+1)} (h_\gamma)]}{\log \gamma} \leq \frac{1}{\cI }. \end{align} Indeed, setting $L=M+1$ in the first relationship in  \eqref{old20} we have  that  as $h \rightarrow \infty$  \begin{align*}\cJ[\tau_{(M+1)}(h)]  = \Exp_{0} [\sigma_{(M+1)}(h)]  \sim \frac{h}{\cI}, \end{align*} and setting $L=1$ in the second relationship in  \eqref{old20} we have  that  $h_\gamma \sim \log \gamma$, which proce \end{proof}

%\begin{corollary}The $M+1$ alarm strategy is asymptotically optimal when $K=2M+1$, in which case as  $\gamma \rightarrow \infty$ we have  \begin{equation*}
%\cJ[\tau_{(M+1)}] \sim \frac{\log \gamma}{\cI} \sim\inf_{T \in \cC_\gamma}\cJ[T].\end{equation*}\end{corollary}

%\subsubsection{ Comparison with the oracle CUSUM rule}
%\subsubsection{The M+1 alarm strategy in groups of sensors}

\begin{remark}
Theorem \ref{theo:alarm} shows that the
first-order asymptotic performance of the  $L^{th}$ alarm is the same for any value of $L$ between $M+1$ and    $|\cB|$.  In view of Remark \ref{rem:one_shot_choice_L}, the proposed choice for $L$ is the smallest possible. This suggest setting  $L=M+1$, \textit{independently of whether the size of the  affected subset is known in advance or not}. 
\end{remark}

 \subsection{Asymptotic analysis of the centralized $L^{th}$ alarm}

 The detection performance of the $L^{th}$ alarm  can be improved significantly if it is  applied  to   \textit{groups} of sensors, instead of individual sensors,  an idea that was suggested in \cite{bay_lai} in the special case $M=1$. Indeed,  let $\cC_{1}, \ldots, \cC_{2M+1}$ be a partition of $[K]$, i.e., 
\begin{align} \label{partition}
\begin{split}
\cC_i \cap \cC_j &= \emptyset \quad \forall  \; 1 \leq i \neq j \leq 2M+1 \\
\text{and} & \quad \bigcup_{i=1}^{2M+1} \cC_i=[K].
\end{split}
\end{align}
%For example, there may be  a second hierarchy level in the network, so that the  fusion center communicates with $2M+1$ nodes and node $i$ has access to the observations from all sensors in $\cB_i$.  
Let $\check{\tau}_{i}(h)$ be the CUSUM stopping time of  the $i^{th}$ group, i.e., 
$\check{\tau}_{i}(h) \equiv \tau_{\cC_i}(h)$, where $\tau_{\cC_i}(h)$ is defined  in \eqref{tB}.  In the following theorem we characterize the first-order  asymptotic  performance of the detection rule that stops when  $M+1$ groups have raised an alarm, i.e., at $\check{\tau}_{(M+1)}(h)$, where 
$$
\check{\tau}_{(1)}(h) \leq \ldots \leq \check{\tau}_{(2M+1)}(h).
$$
For simplicity of presentation, for the following theorem we restrict ourselves to the case that all honest sensors are affected by the change $(\cB= \cN)$, and  $K$ is a multiple of $2M+1$.  
%For the statement and proof of this result, we introduce some additional notation: 
Moreover, we denote by $\check{\cN}$  the subset of honest groups,  by $\{\check{\sigma}_k, k \in \check{\cN}\}$  the alarm times from \textit{only} the honest groups, and  we set $\check{\sigma}_{(1)} \leq \ldots \leq \check{\sigma}_{(M+1)}$.  \\

\begin{theorem}\label{theo:group_alarm}
Suppose that $\cB=\cN$ and that $K$ is a multiple of $2M+1$. Moreover consider a partition  \eqref{partition} in which all subsets have the same size, $K/(2M+1)$.  If we set $h=h_\gamma$  so that 
$\Exp_\infty [\check{\sigma}_{(1)} (h_\gamma)] =\gamma$, then $\check{\tau}_{(M+1)}  (h_\gamma) \in \ccg$. Moreover,  as $\gamma \rightarrow \infty$ 
\begin{align} \label{group_alarm}
\cJN \left[\check{\tau}_{(M+1)}  (h_\gamma) \right] \sim\frac{2M+1}{K} \; 
  \frac{ \log \gamma}{\cI }.
\end{align}
\end{theorem}

\begin{proof}
In the worst-case scenario for both the detection delay and the false alarm rate, there are $M$ groups that contain exactly  one corrupt sensor each, and  all other groups consist of only honest sensors.  Then, similarly to \eqref{repre_oneshot} we have for every $h>0$ that 
\begin{align*}
 \cA[\check{\tau}_{(M+1)}(h)]&=  
  \Exp_{\infty} \left[  \check{\sigma}_{(1)} (h)\right], \\
\quad\cJN [\check{\tau}_{(M+1)}(h)] &= \Exp_{0}^{\cN} \left[  \check{\sigma}_{(M+1)} (h)\right],
 \end{align*}
 and similarly to \eqref{ARL_one_shot} and \eqref{ADD_one_shot} it can be shown that 
\begin{align*}
  \Exp_{\infty} \left[  \check{\sigma}_{(1)} (h) \right] &\sim \Theta(1) \, e^{h}, \\
  \Exp_{0}^{\cN} \left[  \check{\sigma}_{(M+1)}(h) \right] &\sim \frac{ 2M+1}{K}  \; \frac{h}{\cI},
 \end{align*}
which implies \eqref{group_alarm}.\\
 \end{proof}

\begin{remark}
A comparison of \eqref{alarm} and \eqref{group_alarm} reveals that,  under the conditions of Theorem \ref{theo:group_alarm}, the
centralized $(M+1)$-alarm is asymptotically more efficient than  the decentralized $(M+1)$-alarm (recall Definition \ref{def2}),   since its   first-order asymptotic  detection delay  is  $K/(2M+1)$ smaller. We will see in the next sections that we can achieve even better asymptotic performance with the other two procedures under consideration. 
\end{remark}

\begin{remark}
The decentralized and centralized \textit{second}-alarm,  $\tau_{(2)}$ and $\check{\tau}_{(2)}$, were  proposed in \cite{bay_lai}, in the case that  all honest sensors are affected, and asymptotic upper bounds were obtained for the performance of these procedures.  Setting $L=2$ and $M=1$ in Theorems \ref{theo:alarm} and
 \ref{theo:group_alarm}  we improve upon these results by  characterizing the  performance of $\tau_{(2)}$ and $\check{\tau}_{(2)}$ up to a first-order asymptotic approximation.
\end{remark}

\subsection{Asymptotic analysis of the voting rule}
We now study the   asymptotic performance of the 
voting rule, $T_{L}$, that was defined in \eqref{general_voting}.\\

\begin{theorem} \label{th:voting_byzantine_optimality}
Suppose that $L > M$.  If $h=h_\gamma$ is  so that 
$\Exp_\infty[S_{L-M}(h_\gamma)]=\gamma$, then 
$T_{L} (h_\gamma) \in \ccg$. If also  $ |\cB| \geq L$, then  as $\gamma \rightarrow \infty$ we have
 \begin{equation}   \label{upper44}
\cJ\left[T_{L}(h_\gamma)\right] \sim \frac{\log \gamma}{ (L-M) \, \cI}.
\end{equation}
 \end{theorem}

\begin{proof}
From \eqref{repre_voting}  we know that  for every $h>0$
\begin{align*}
\cA\left[T_{L}(h)\right] &= \Exp_\infty\left[S_{L-M}(h)\right],\\ \cJ\left[T_{L}(h)\right] &= \Exp_{0}^{\cB}[S_{L}(h)].
\end{align*}
From   \eqref{ADD_one_shot} we have  that 
$\Exp_{0}^{\cB} \left[S_{L}(h) \right]  \sim h/\cI $
 as $h \rightarrow \infty$,  and  from \eqref{ARL_majority_approximation}  that  if
$h=h_\gamma$ is  so that 
$\Exp_\infty[S_{L-M}(h_\gamma)]=\gamma$, then
 $h_\gamma  \sim  (\log \gamma)/ (L-M)$  as $\gamma \rightarrow \infty$. This implies  that as $\gamma \rightarrow \infty$ 
  \begin{equation}  \label{upper444}
\Exp_{0}^{\cB} \left[S_{L}(h_\gamma) \right] \sim \frac{\log \gamma}{ (L-M) \, \cI}
\end{equation} 
and    completes the proof. \\
\end{proof}

\begin{remark}
From \eqref{upper44} it follows that the asymptotic worse-case detection delay of the voting rule, $T_L$, is  decreasing in $L$, which implies that $L$ should  be as large as possible. Since $L$ must be at most equal to the  size of the affected subset, $|\cB|$, this means that the  selection of $L$ in the family of voting rules   depends heavily on  prior knowledge regarding $|\cB|$. Indeed, in the  absence of any information,  $L$ must be set equal to $M+1$, and the resulting  first-order  asymptotic performance is the same as that of the $(M+1)$-alarm, $\tau_{(M+1)}$.  On the other hand,  in the ideal case that $|\cB|$ is known in advance,  the asymptotic approximation \eqref{upper44} suggests setting $L=|\cB|$,  in which  case 
the resulting  first-order asymptotic performance   is $|\cB|-M$ times smaller than that of the $(M+1)$-alarm.  We will now see that this asymptotic performance is achieved by \textit{Low-Sum-CUSUM, without prior knowledge of $|\cB|$}. 
\end{remark}

\subsection{Asymptotic performance of Low-Sum-CUSUM}

We now turn to the asymptotic analysis of  \textit{Low-Sum-CUSUM}, that was defined in \eqref{general_low_sum_cusum}. 

%We will see that this rule achieves the asymptotic performance \eqref{upper4444} \textit{without any prior knowledge regarding the size of the affected subset}. \\

%As we mentioned before, similarly to the $M+1$ alarm, this rule assumes knowledge of the number of corrupt sensors, but not of the number of  affected sensors. However, unlike the  $M+1$ alarm,$|\cN|$-low-Sum-CUSUM can only be implemented in a centralized fashion. 

%\noindent \textit{\underline{Remarks}}: It becomes clear from this proposition  that this detection rule, as well, requires the assumption that   $K \geq 2M+1$. 
%2) Similar arguments would apply if we use, instead of the sum, a statistic of the form $ H(Y_t^{(1)}, \ldots, Y_t^{(K-M)})$, where $H:\bR_{+}^{K-M} \rightarrow [0, \infty)$ is a non-negative function that is increasing in each of its arguments. \\

\begin{theorem}
Suppose $1 \leq L \leq K-M$.   If $h=h_\gamma$ is so that  $\Exp_{\infty} [ \widetilde{S}_{L}(h_\gamma) ]=\gamma$, then $\widetilde{T}_{L}(h_\gamma) \in \ccg$. 
If also $L > K-|\cB|$, then   as $\gamma \rightarrow \infty$ we have that
\begin{align}  \label{upper_lowsumcusum}
\cJ \left[\widetilde{T}_{L}(h_\gamma) \right] 
& \sim  \frac{L}{K-M}  \,   \frac{\log \gamma}{( |\cB| -(K-L))\, \cI}.
 \end{align}
%When in particular $\cB=\cN$, 
%\begin{equation}  \label{upper_lowsumcusum_all_affected}
%\cJ \left[\widetilde{T}_{L}(h_\gamma) \right]\sim
 %\frac{L}{|\cN|}  \,   \frac{\log \gamma}{(L-M)\, \cI}.
%\end{equation}
\end{theorem}

\begin{proof}
From \eqref{repre_lowsumcusum}  we have  for every $h>0$ that
\begin{align*}
\cA\left[ \widetilde{T}_{L}(h) \right] &= \Exp_{\infty} \left[  \widetilde{S}_{L}(h)\right], \\
\cJ \left[\widetilde{T}_{L}(h) \right] &=  \Exp_{0}^{\cB} \left[ \widetilde{S}_{L-M}(h) \right].
\end{align*}
From  \eqref{upper6}  it follows that as $h \rightarrow \infty$ 
$$ 
\Exp_{0}^{\cB} \left[ \widetilde{S}_{L-M}(h) \right] \sim 
\frac{h}{(L-M - |\cN|+ |\cB|)\, \cI},
$$   
and from Theorem \ref{th:lowsumcusum}  that if  $h=h_\gamma$ is selected so that  $ \Exp_{\infty} [\widetilde{S}_{L}(h_\gamma)]=\gamma$, then $h_\gamma \sim (L/ |\cN|) \, \log \gamma$. Thus, setting $h=h_\gamma$ in the previous relationship and using the fact that $|\cN|=K-M$, we obtain  \eqref{upper_lowsumcusum}.\\
\end{proof}

The asymptotic performance  \eqref{upper_lowsumcusum}  of 
\textit{Low-Sum-CUSUM} is decreasing in $L$, which implies that $L$  should be selected equal to its largest possible value, $K-M$. Therefore, the proposed value of $L$ for \textit{Low-Sum-CUSUM} does \textit{not}  require knowledge of the size of the affected subset. The following corollary describes the resulting first-order asymptotic performance.\\

\begin{corollary}
For any $\gamma \geq 1$, if $h=h_\gamma$ is so that  $\Exp_{\infty} [ \widetilde{S}_{K-M}(h_\gamma)]=\gamma$, then $\widetilde{T}_{K-M}(h_\gamma) \in \ccg$. If  $|\cB| \geq M+1$, then as $\gamma \rightarrow \infty$ we have 
\begin{align}  \label{asymptotic_performance_lowsumcusm}
 \cJ \left[\widetilde{T}_{K-M}(h_\gamma) \right] \sim
  \frac{\log \gamma}{(|\cB|-M)\, \cI}.
\end{align}
\end{corollary}

\subsection{Discussion} \label{subsec:discuss}
We now summarize the main results of this section. First of all, in the special case that  $K-M=M+1$, the only possible choice for $L$ for  all three families under consideration is  $M+1$ and, for any given false alarm rate,    the $(M+1)$-alarm dominates the corresponding voting rule, $T_{M+1}$, and the latter  dominates the corresponding  \textit{Low-Sum-CUSUM}, $\widetilde{T}_{M+1}$.

When  $K-M>M+1$,  the proposed values for $L$
are $M+1$ and $K-M$ for  $\tau_{(L)}$  and  $\widetilde{T}_L$, respectively, thus, they  do not require  knowledge of  the size of the affected subset, $|\cB|$, and  the  first-order asymptotic detection delay of \textit{Low-Sum-CUSUM}, $\widetilde{T}_{K-M}$ is $|\cB|-M$ smaller than that of the $(M+1)$-alarm.  On  the other hand, the selection of $L$ for the voting rule, $T_{L}$, depends on prior knowledge regarding  $|\cB|$. However, even in the ideal case that $|\cB|$ is  known in advance, the asymptotic performance of the voting rule with $L=|\cB|$, $T_{|\cB|}$, is  the same as that of \textit{Low-Sum-CUSUM}, $\widetilde{T}_{K-M}$. % In the following section, we conduct a simulation study that reveals that  even if it is known in advance that all sensors are affected ($\cB=\cN$), $\widetilde{T}_{K-M}$  performs   better  than  $T_{K-M}$.

%Before we do so,  it is important to stress that the three procedures under consideration require different kinds of information and can be implemented with different levels of communication activity from the sensors to the fusion center. In particular,  $\tau_{(M+1)}$ and $\widetilde{T}_{|\cN|}$ require knowledge of only the number of corrupt sensors, whereas  $T_{|\cB|}$ requires knowledge of only the number of affected sensors. Moreover, the $(M+1)^{th}$ alarm strategy is extremely simple, as it requires transmission of at most one bit from each sensor, whereas the voting rule, $T_{|\cB|}$, requires one-bit transmissions from each sensor to the fusion center whenever its local CUSUM statistic becomes larger than $h$ and whenever it becomes smaller than $h$. On the other hand, $\widetilde{T}_{|\cN|}$ requires transmission of the full value of the local CUSUM statistic from each sensor at each time. 

Finally, we  found that although the centralized $(M+1)$-alarm achieves much better asymptotic performance than the corresponding decentralized $(M+1)$-alarm, it is always asymptotically less efficient than  \textit{Low-Sum-CUSUM}. \\  %This conclusion will also be supported by the simulation study in the next section. \\
   
%When \eqref{dicho} does not hold,  a similar, non-asymptotic ordering is not available, but our asymptotic analysis reveals that the previous ordering of the  proposed procedures  is completely reversed. Indeed, while the $M+1$ alarm  preserves the same asymptotic performance as in the previous case,  the first-order asymptotic performance of both $T_{|\cB|}$ and  $\widetilde{T}_{|\cN|}$, is $(\log \gamma)/ ((|\cB|-M) \, \cI)$. Therefore, the larger the difference between the number of affected and corrupt sensors,  the higher the loss inflicted by using the $M+1$ alarm strategy.  %When the number of both the honest and affected sensors is known, the two efficient procedures, 

%While $\widetilde{T}_{|\cN|}$ and $T_{|\cB|}$ have the same first-order asymptotic performance,  a more refined asymptotic analysis can show that the second-order term in the performance of $\widetilde{T}_{|\cN|}$ is  $\calo(\log \log \gamma)$, whereas that of  $T_{|\cB|}$ is $\calo(\sqrt{\log \gamma})$. Therefore,  we expect that  $\widetilde{T}_{|\cN|}$ should perform better than $T_{|\cB|}$ in practice. \\

%However,  However, if bandwidth constraints are of interest, then we believe that simple modifications of $\widetilde{T}_{|\cN|}$ can be more efficient in practice. 

\begin{remark}
In Section \ref{sec:oracle} we saw that, in the absence of corrupt sensors,  it is possible to achieve the optimal asymptotic performance \eqref{opt_perf}  for any affected subset $\cB \subseteq \cN$,  up to a first-order asymptotic approximation, or even up to a constant term,  even if there is absolutely no information about the  affected subset. This  is not the case in the presence of corrupt sensors, at least for the detection rules that we study here. Indeed, comparing \eqref{cusum_perf} with  \eqref{asymptotic_performance_lowsumcusm} we can see that  the best  first-order asymptotic performance that can be achieved by the proposed procedures is  the same  as  that of a centralized CUSUM rule that utilizes only  $|\cB|-M$, not $|\cB|$, honest, affected senors. 
\end{remark}

%When both $\sigma_{\cB}$ and $\tau_{(M+1)}$, $T_{|\cB|}$ are designed to satisfy the false alarm constraint with equality, then  as $\gamma \rightarrow \infty$ we have for every  $\cB \subseteq \cN$$$\frac{\cJ [\tau_{(M+1)}]}{\cJ [\sigma_{\cB}]} \longrightarrow |\cB|, \quad  \frac{\cJ [T_{|\cB|}]}{\cJ [\sigma_{\cB}]} \longrightarrow \frac{|\cB|-M}{|\cB|}, $$whereas$$\limsup  \frac{\cJ [\widetilde{T}_{|\cN|}]}{\cJ [\sigma_{\cB}]} \leq \frac{|\cB|}{|\cB|-M},$$and at least when $|\cB|=|\cN|=M+1$, $$ \frac{\cJ [\widetilde{T}_{|\cN|}]}{\cJ [\tau_{\cB}]} \rightarrow M+1.$$
%While $T_{|\cB|}$ and $\widetilde{T}_{|\cN|}$ have the same asymptotic performance, at least for every $|\cB| \geq |\cN|-M$, it should be noted that  $T_{|\cB|}$ assumes knowledge of the size of the affected subset. 
%On the other hand, it can be implemented in a decentralized fashion. 

\section{Simulation Experiments} \label{sec:simulation}
We now  illustrate our theoretical findings in the previous section with two simulation studies where   all honest sensors are  normally distributed with variance 1 and mean $0$ before the change and $1$ after the change, i.e., $f=\cN(0,1)$ and $g=\cN(1,1)$.  That is, all honest sensors are  affected by the change ($\cB=\cN$).

In the first simulation study, there are $M=2$ corrupt and $|\cN|=3$ honest sensors, therefore $K=2M+1$.   In Figures 1(a),(c) we compare the performance of the $(M+1)$-alarm, $\tau_{(M+1)}$, the voting rule, $T_{M+1}$, and \textit{Low-Sum-CUSUM}, $\widetilde{T}_{M+1}$. As predicted by Corollary \ref{corol},  we see that  $\tau_{(M+1)}$ dominates $T_{M+1}$, and $T_{M+1}$ dominates $\widetilde{T}_{M+1}$ for any level of the false alarm rate.

 In the second simulation study, we have $M=1$ corrupt and $|\cN|=5$ honest sensors, thus,  
$K>2M+1$.   In Figures 1(b),(d),  we  compare the performance of $\tau_{(M+1)}$, $\widetilde{T}_{K-M}$, and the voting rule, $T_{K-M}$, which assumes knowledge of the fact that all honest sensors are affected. Moreover, we consider the centralized version of the $(M+1)$-alarm studied in Section \ref{sec:asymptotics}. As expected from our asymptotic results,  we observe that the centralized  $(M+1)$-alarm performs  better than the decentralized $(M+1)$-alarm, $\tau_{(M+1)}$, but worse compared to  the voting rule  $T_{K-M}$. A  more interesting observation is that  \textit{Low-Sum-CUSUM}, $\widetilde{T}_{K-M}$,  performs better than the voting rule, $T_{K-M}$, despite the fact that they have the same first-order asymptotic performance. 
 
%  which does not make any assumption regarding the size of the affected subset. A  more interesting observation is that $T_{|\cB|}$ performs worse than  $\widetilde{T}_{K-M}$, despite the fact that it assumes knowledge of the size of the affected subset, 

\begin{figure}
\centering
\begin{tabular}{cc}
\includegraphics[width=0.45\linewidth]{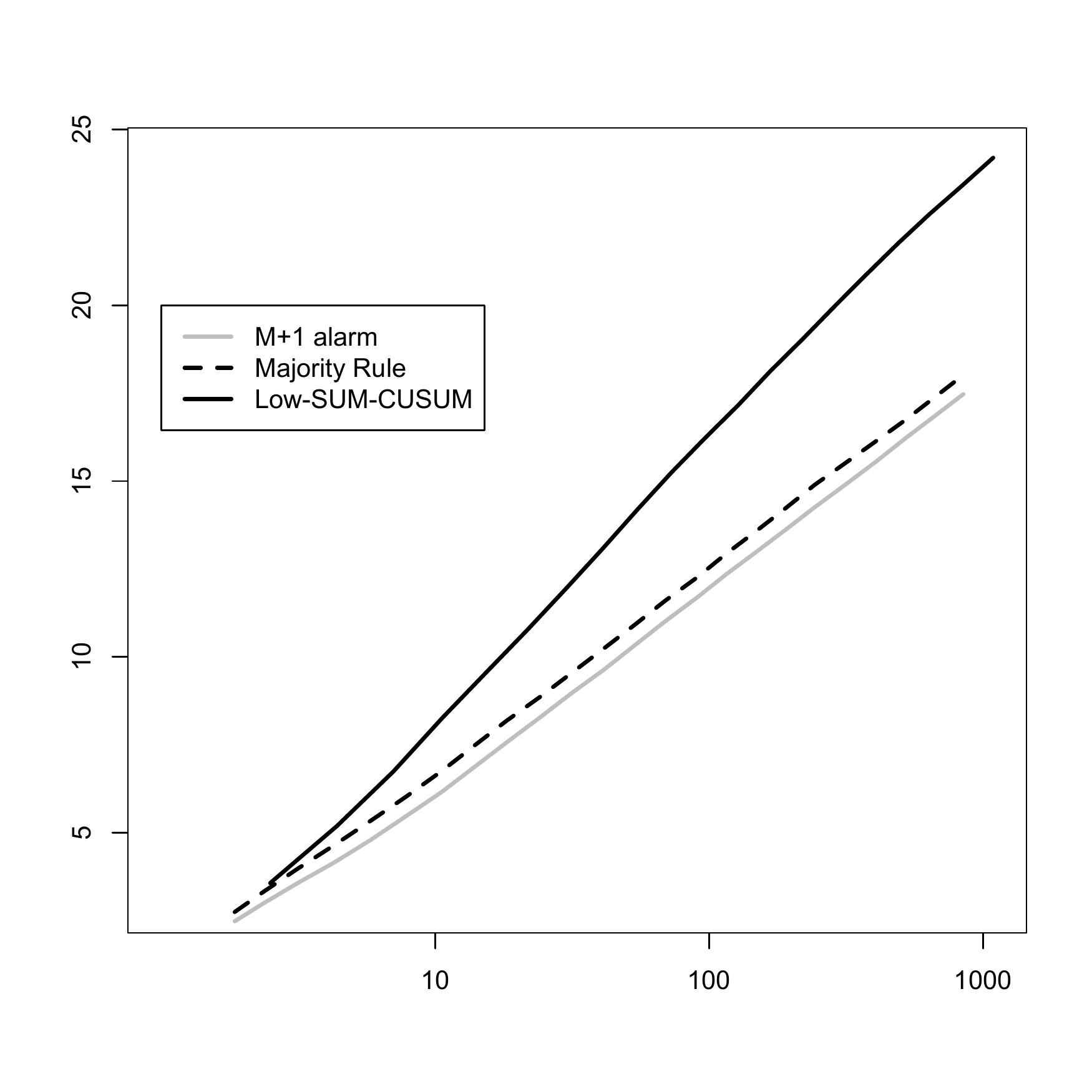}  
& 
\includegraphics[width=0.45\linewidth]{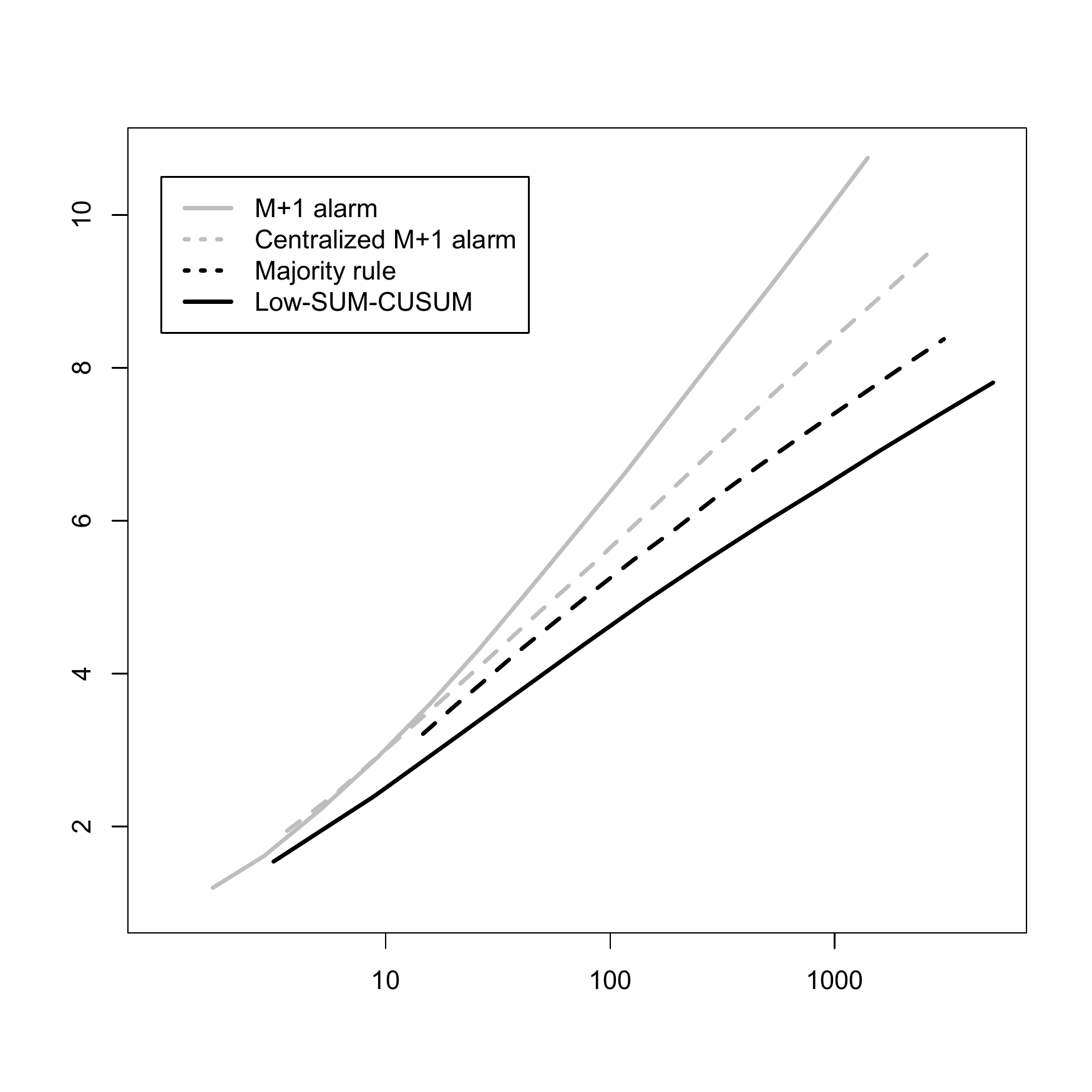}  \\
(a) $K=5$,  $M=2$ 
&
(b) $K=6$, $M=1$  \\
\includegraphics[width=0.45\linewidth]{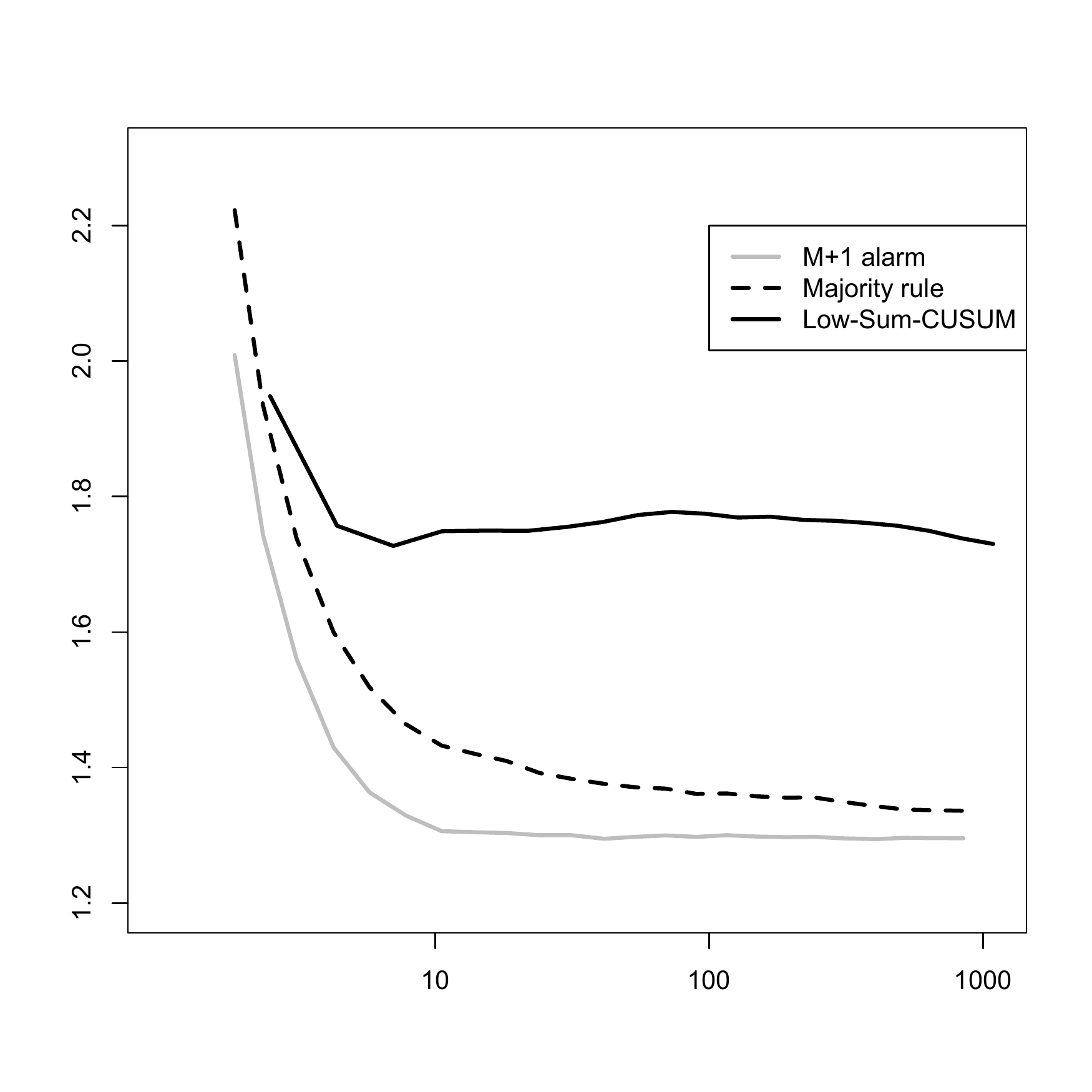}  
&
\includegraphics[width=0.45\linewidth]{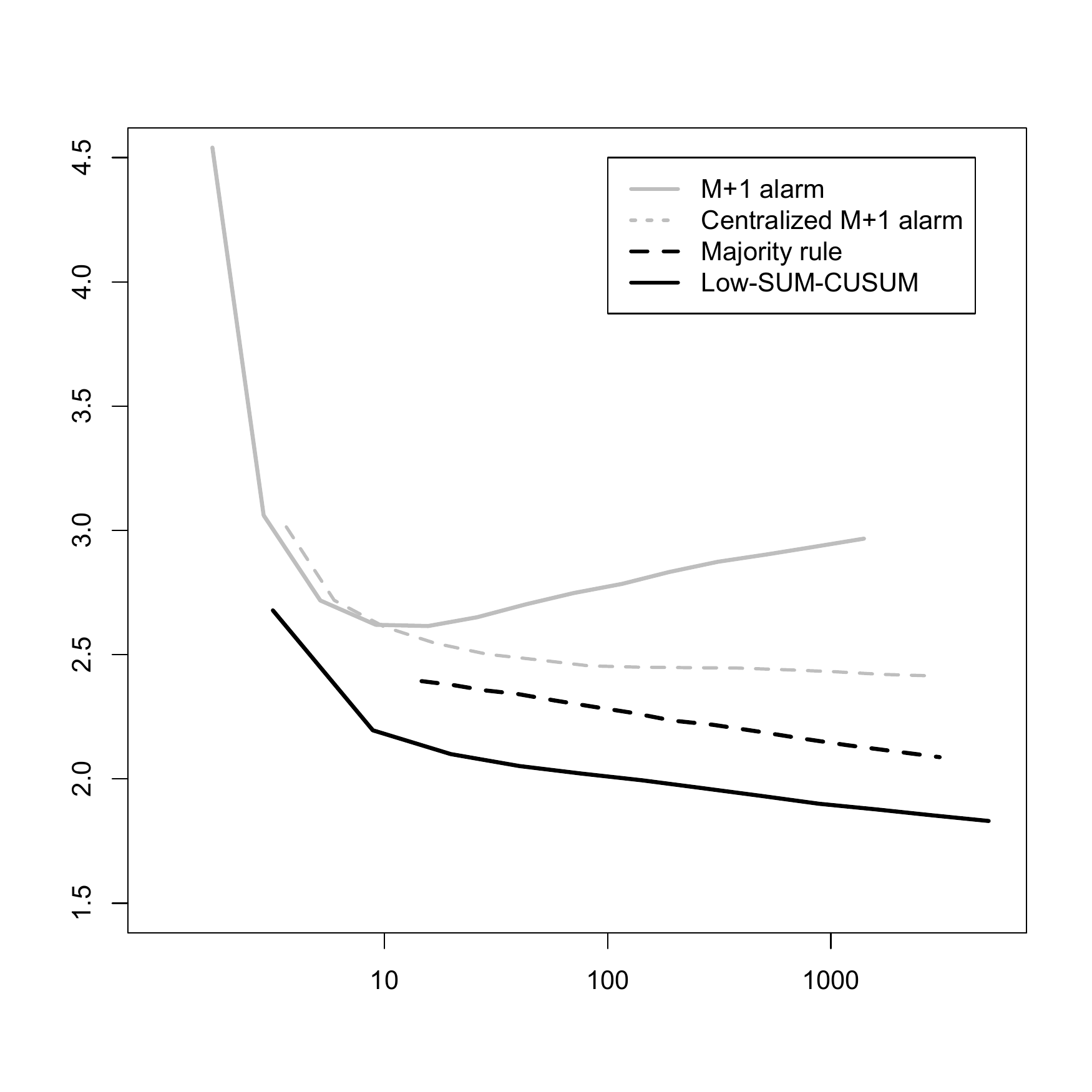} \\
(c) $K=5$,  $M=2$ 
& 
(d) $K=6$, $M=1$ \\
\end{tabular}
\caption{ In all  graphs, the horizontal axis corresponds to the worst-case expected time to false alarm (in log scale), i.e., $\log \gamma$. In (a) and (b), the vertical axis corresponds to the worst-case detection delay, whereas in (c) and (d) to  a normalized version of the latter, i.e., divided by $(\log \gamma)/ (K-2M) \cI)$. 
 In all graphs, the solid, dark lines corresponds to the 
Low-Sum-CUSUM,  the dashed line to the voting rule, and the gray, solid line to the  $M+1$ alarm.  The gray, dotted line in (b) and (d) corresponds to the centralized version of the  $M+1$ alarm.}
\end{figure}

\section{Conclusions} \label{sec:conclude}
In the classical multisensor sequential change-detection problem,   data are collected sequentially from a number of sensors, and the goal is to detect  quickly and accurately a change that is perceived by only an unknown subset of these sensors, while the observations in all non-affected sensors continue following their initial distribution. In this classical setup, any model mis-specification in the non-affected sensors  is  ignored and underestimated. 
In this paper we considered a different  formulation of this  problem, in which  at most $M$  unknown sensors are considered to be unreliable and are treated as  if they are controlled by an adversary.
This  generalizes the  formulation in  \cite{bay_lai}, 
in that we allow for more than one corrupt sensors, i.e., $M\geq 1$, and we assume that  the  subset, $\cB$, of  honest sensors  affected by the change is unknown. We proposed three families of detection rules  that  were evaluated under a generalization of Lorden's criterion,   \textit{in the worst case scenario regarding the strategy of the adversary, when  there are exactly $M$ corrupt sensors}. However, as in the classical multichannel setup, we did not adopt a worst case approach with respect to the affected subset of sensors, $\cB$. 
% Similarly to \cite{bay_lai},  we have focused on the case that the pre-change distribution is the same in all honest sensors, and the post-change distribution is the same in all honest sensors affected by the change. However, in the present work  we   consider a much more general setup compared to  \cite{bay_lai}, as we allow for more than one corrupt sensors, i.e., $M\geq 1$, and we assume that  the  subset, $\cB$, of  honest sensors  affected by the change is unknown. 

% To this end, we focus on the case that the size of the affected subset  $|\cB|$ is larger than $M$, i.e., $|\cB| \geq M+1$.  However, we should stress that while we take a worst case approach with respect to the number of corrupt sensors, we do not in general  require any knowledge regarding the affected subset $\cB$.

The first proposed procedure  stops as soon as $M+1$ local CUSUM statistics have crossed a common threshold. This procedure is shown to be the best, in an exact sense, than all other proposed rules in the special that there are $M+1$ honest sensors, all affected by the change.  Setting $M=1$ reveals that second alarm, proposed  in \cite{bay_lai}, is the best rule among the ones considered here   in the special case of  $K=3$ sensors.  In the general case that the number of honest sensors exceeds  the number of corrupt sensors by more than 1, the previous scheme  can be very inefficient, as its  first-order asymptotic performance is shown to be independent of the size of the affected subset. We show that it is possible to achieve  much better performance with a novel procedure, which also  does not require knowledge of the true size of the affected subset. This procedure  stops as soon as the  sum of the smallest $K-M$ local CUSUM statistics crosses a fixed threshold, and we refer to it  as  \textit{Low-Sum-CUSUM}. We show that its  first-order asymptotic performance is the same as that of a centralized CUSUM  that relies on  $|\cB|-M$ honest sensors, all affected by the change. We conjecture that this is the best possible first-order asymptotic performance in the presence of $M$ corrupt sensors,  but the proof of this result is an open problem.

These results are not relevant only for the design of sequential change-detection rules in an adversarial setup, but can also be useful for the  ``robustification'' of existing multichannel procedures.  Indeed, when there is a  large number of sensors, $K$, and a non-trivial lower bound $Q$ on the size of the affected subset, our results suggest that   \textit{Low-Sum-CUSUM} with a small $M\leq Q$ can lead to more robust   behavior with  a relatively small price in efficiency.

The procedures under consideration have low computational complexity. The heavier  communication requirements  from the sensors to the fusion center are imposed by   \textit{Low-Sum-CUSUM}, which requires that each sensor transmits the value of its local CUSUM statistic at each time. It is possible to design bandwidth-efficient modifications of this scheme, thresholding  each local CUSUM statistic below and communicating only when its value is above this threshold  \cite{mei_sympo}. It is also possible  to design energy-efficient modifications of   \textit{Low-Sum-CUSUM} \cite{7105922}, where the local CUSUM statistics do not need to be observed continuously at the sensors.   

Similarly to \cite{bay_lai},  we have focused on the case that the pre-change distribution is the same in all honest sensors, and the post-change distribution is the same in all honest sensors affected by the change. Another interesting generalization of our work is in the  non-homogeneous setup. Finally, our setup is clearly relevant in security related applications. An interesting  alternative approach in this context is a  game-theoretic  formulation.
% sensors, that is when the  pre-change and/or post-change distributions may vary across  the various honest sensors.

% in order to lighten the notation, however  our results can be generalized in a straightforward fashion in this case.
% for the $M+1$ alarm and the low-SUM-CUSUM. The same can be done for the voting rule, however the performance of this rule can b can be improve

% where the goal is to design a procedure that corresponds to a  worst-case strategy of the adversary.  
%understand the worst-possible str problem formulation in this context would be to  consider an  adversarial setup where the detector tries to understand the corrupt (or malicious) sensors, something that  poses limitations in the strategy of the adversary. 
%In this context, a game-theoretic  formulation is required, which has been considered before\textcolor{red}{add citations}, but in general it is an open problem. 

\section*{Appendix}
In this Appendix, all random variables are defined on some probability space $(\Omega, \cF, \Pro)$. 

\begin{lemma} \label{newlemma}
Let $X_g, Y_g$ be independent, non-negative, integer-valued random variables, parametrized by some positive constant $g$.
\begin{enumerate}
\item[(i)] If $\Pro(X_g \leq t) \leq t/g$ for every $t=0,1, 2, \ldots.$, then   $\Exp[X_g] \geq g/2$. 
\item[(ii)] If additionally $\Exp[Y^2_g]=o(g)$ as $g \rightarrow \infty$,  then  $\Exp[Y_g]= \Exp[\min \{X_g, Y_g\}] +o(1)$ as $g \rightarrow \infty$.
\end{enumerate}
\end{lemma}

\begin{proof}
(i) From the assumption of the Lemma and the non-negativity of probability  it follows that 
  $\Pro(X_g>t) \geq (1-t/g)^+$. Therefore, 
\begin{align*}
\Exp[X_g]&=1 + \sum_{t =1}^{\infty} \Pro(X_g>t) \\
&\geq 1 + \sum_{t =1}^{\lfloor g \rfloor}  (1-t/g) \\
&= 1+ \lfloor g \rfloor - \frac{\lfloor g \rfloor (1+\lfloor g \rfloor)}{2g} \\
&=  (1+ \lfloor g \rfloor) \left( 1-\frac{ \lfloor g \rfloor}{2g} \right) \geq g/2.
\end{align*}

(ii)  From the independence of $X_g$ and $Y_g$ it follows that 
\begin{align*}
\Pro(X_g>t, Y_g>t) &= \Pro(X_g>t) \, \Pro(Y_g>t) \\
&= \Pro(Y_g>t)- 
\Pro(X_g\leq t) \cdot \Pro(Y_g>t)
\end{align*}
for every $t=0,1, \ldots$. Therefore, 
$$
\Exp[\min \{X_g, Y_g\}] = \Exp[Y_g]-  \sum_{t =0}^{\infty} \Pro(X_g\leq t) \, \Pro(Y_g>t),
$$
and it suffices to show that the second term on the right-hand side goes to 0 as $g \rightarrow \infty$. Indeed, by the assumption on the cdf of $X_g$ we have 
\begin{align*}
 \sum_{t =0}^{\infty} \Pro(X_g\leq t) \, \Pro(Y_g>t) &\leq
\frac{1}{g}  \sum_{t =0}^{\infty} t \,  \Pro(Y_g>t)  \\
&= \frac{1}{2 g} \, \Exp[Y_g^2] 
\end{align*}
and the upper bound goes to 0 by the assumption on the second moment on $Y_g$. \\
\end{proof}

\begin{lemma} \label{lem:asy}
Let $\xi_{1}, \ldots, \xi_{|\cN|}$ be independent, exponential random variables with mean 1, and consider the order-statistics  $\xi_{(1)} \leq  \ldots \leq \xi_{(|\cN|)}$.  
   For each $1 \leq L \leq |\cN|$ and   $h>0$ set 
\begin{align}  \label{GG}
 G_L(h) \equiv \Pro \left( \sum_{k =1}^L  \xi_{(k)}  >h \right).
  \end{align}
Then,  as  $h \rightarrow \infty$
\begin{align}  \label{GG}
 G_L(h) \sim  \Theta(1) \, \exp\{ -(|\cN|/L)  h\} .
  \end{align}

\end{lemma}

\begin{proof}
Set $\xi_0=0$. From the so-called R\'enyi representation \cite{Renyi53} it follows  that the spacings 
$\eta_j=\xi_{(j)}- \xi_{(j-1)}$, $1 \leq j \leq |\cN|$ 
are independent, exponential random variables such that  
$$\eta_j \sim \cE(|\cN|-j +1)\sim \frac{\xi_j}{|\cN|-j +1}, \quad 1 \leq j \leq |\cN|.
$$ 
Then, 
$$ \sum_{k =1}^L  \xi_{(k)} =  \sum_{k =1}^L  \sum_{j=1}^{k} \eta_j=  \sum_{j =1}^L  \sum_{k=j}^{L} \eta_j=  \sum_{j =1}^L (L-j+1)  \eta_j,$$
and consequently we have 
$$ \sum_{k =1}^L  \xi_{(k)}  \; \overset{D}= \;   \sum_{j =1}^L  \lambda_j \; \xi_j,
$$
where 
$$ \lambda_j= \frac{L-j+1}{|\cN|-j +1}, \quad 1 \leq j \leq L.$$
Therefore, when $L<|\cN|$,  we have $\lambda_1 > \ldots> \lambda_{|\cN|}$ and from \cite[Lemma 11.3.1]{Nagaraja2006} it follows that there are 
positive numbers  $C_L$, $1 \leq L < |\cN|$ so that
$$
G_{L}(h)=   \sum_{j =1}^L C_j e^{-h/\lambda_j } \sim C_{1} e^{ -h/ \lambda_{1}} 
 =  C_{1}  e^{ -(|\cN|/L) h}.
$$
%When $L=|\cN|$,  $\lambda_j=1$ for every $1 \leq j \leq |\cN|$. Since $\sum_{j =1}^{|\cN|}  \xi_j$  is an   Erlang random variable with parameter $|\cN|$, we conclude that 
%$$G_{|\cN|}(h)= e^{-h} \sum_{j=0}^{|\cN|-1} \frac{h^j}{ j!} \sim e^{-h} \, \frac{h^{|\cN|-1}}{ 
%(|\cN|-1)!}. $$

\end{proof}

 \begin{lemma}[\cite{farrell1964}] \label{lem:simultaneous}
Let $\{U_t^i\}_{t \in \bN}$, $1 \leq i \leq M$ be possibly dependent, random walks with positive drifts. That is,  for any given $1 \leq i \leq M$, the  increments $\{U_t^i-U_{t-1}^i\}_{t \in \bN}$ are integrable, iid random variables with mean $\mu_i>0$. Consider the family of stopping times 
$$
T_b=\inf\left\{t \in \bN: \min_{1 \leq i \leq M} U_t^i \geq b\right\},
$$
where $b>0$.  Then,  as $b \rightarrow \infty$ we have 
$$
\Exp[T_b] \sim \frac{b}{\min_{1 \leq i \leq M} \mu_i}.
$$
\end{lemma}

\begin{proof}
This is a special case of \cite[Theorem 3]{farrell1964}
\end{proof}

\bibliographystyle{ieeetr}
\bibliography{multisensor}

\end{document}